\documentclass{article}
\usepackage{amsmath,amsfonts,amsthm,amssymb,amscd,color,xcolor}
\usepackage{hyperref}

\colorlet{darkblue}{blue!50!black}

\hypersetup{
    colorlinks,%
    citecolor=darkblue,%
    filecolor=red,%
    linkcolor=darkblue,%
    urlcolor=magenta,%
    pdfnewwindow=true,%
    pdfstartview={FitH}
}

\newcommand{\p}{\partial}
\newcommand{\e}{\varepsilon}

\newcommand{\R}{{\mathbb R}}
\newcommand{\IP}{{\mathbb P}}
\newcommand{\Q}{{\mathbb Q}}
\newcommand{\Z}{{\mathbb Z}}
\newcommand{\E}{{\mathbb E}}

\newcommand{\N}{{\mathbb N}}

\newcommand{\PPhi}{{\boldsymbol\varPhi}}
\newcommand{\PPsi}{{\boldsymbol\varPsi}}

\newcommand{\aA}{{\cal A}}
\newcommand{\BB}{{\cal B}}
\newcommand{\CC}{{\cal C}}

\newcommand{\FF}{{\cal F}}

\newcommand{\HH}{{\cal H}}
\newcommand{\KK}{{\cal K}}

\newcommand{\MM}{{\cal M}}

\newcommand{\OO}{{\cal O}}

\newcommand{\TT}{{\cal T}}

\newcommand{\XX}{{\cal X}}
\newcommand{\YY}{{\cal Y}}

\newcommand{\BBBB}{{\mathfrak B}}

\newcommand{\kkk}{{\boldsymbol{\mathit k}}}

\newcommand{\diam}{\mathop{\rm diam}\nolimits}

\newcommand{\area}{\mathop{\rm Area}}

\theoremstyle{plain}
\newtheorem*{mta}{Theorem A}
\newtheorem*{mtb}{Theorem B}
\newtheorem{theorem}{Theorem}[section]
\newtheorem{lemma}[theorem]{Lemma}
\newtheorem{proposition}[theorem]{Proposition}
\newtheorem{corollary}[theorem]{Corollary}
\theoremstyle{definition}
\newtheorem{definition}[theorem]{Definition}
\newtheorem{condition}[theorem]{Condition}
\theoremstyle{remark}
\newtheorem{remark}[theorem]{Remark}

\newtheorem*{example*}{Example}

\numberwithin{equation}{section}

\begin{document}
\author{Armen Shirikyan\footnote{Department of Mathematics, University of Cergy--Pontoise, CNRS UMR 8088, 2 avenue Adolphe Chauvin, 95302 Cergy--Pontoise, France; E-mail: Armen.Shirikyan@u-cergy.fr} \and Sergey Zelik\footnote{Department of Mathematics, University of Surrey, Guildford GU2 7XH, UK; E-mail: S.Zelik@surrey.ac.uk}}
\title{Exponential attractors for random dynamical systems and applications}
\date{\today}
\maketitle
\begin{abstract}
The paper is devoted to constructing a random exponential attractor for some classes of stochastic PDE's. We first prove the existence of an exponential attractor for abstract random dynamical systems and study its dependence on a parameter and then apply these results to a nonlinear reaction-diffusion system with a random perturbation. We show, in particular, that the attractors can be constructed in such a way that the symmetric distance between the attractors for stochastic and deterministic problems goes to zero with the amplitude of the random perturbation. 


\end{abstract}

\tableofcontents

%\setcounter{section}{-1}

%\newpage
\section{Introduction}
\label{s1}
The theory of attractors for partial differential equations (PDE's) has been developed intensively since late seventies of the last century. It is by now well known that many  dissipative PDE possesses a minimal attractor, even if the Cauchy problem is not known to be well posed. Moreover, one can establish explicit upper and lower  bounds for the dimension of a minimal attractor. A comprehensive presentation of the theory of attractors can be found in~\cite{CV2002,SY2002}. Similar results were also proved in the case of random dynamical systems (RDS) generated by stochastic PDE's, such as the Navier--Stokes system or reaction-diffusion equations with random perturbations; see~\cite{CF-94,CDF-97}. A drawback of the theory of attractors is that, in general, it is impossible to have any estimate for the rate of convergence to the minimal attractor. Furthermore, in the case of RDS, the attraction property holds when the initial time goes to~$-\infty$, whereas one is usually interested in the large-time asymptotics of solutions for the Cauchy problem with a fixed initial time. To remedy these shortcomings, a concept of exponential attractors was suggested in~\cite{EFNT1994} for deterministic problems. In contrast to the attractors discussed above, they do not possess any minimality property, but still  have a finite fractal dimension and, moreover, attract trajectories exponentially fast. We refer the reader to the  review paper~\cite{MZ-2008} (and the references therein) for a detailed account of the results on exponential attractors obtained so far. 

The aim of this article is to construct finite-dimensional exponential attractors for some classes of RDS and then to show that the general results are applicable in the case of reaction--diffusion equations. To be precise, let us consider from the very beginning the following problem in a bounded domain $D\subset\R^n$ with a smooth boundary~$\p D$:
\begin{align} 
\dot u-a\Delta u+f(u)&=h(x)+\eta(t,x), \label{1.1}\\
u\bigr|_{\p D}&=0,\label{1.2}\\
u(0,x)&=u_0(x). \label{1.3}
\end{align}
Here $u=(u_1,\dots,u_k)^t$ is an unknown vector function, $a$ is a $k\times k$ matrix such that $a+a^t>0$, $f\in C^2(\R^k,\R^k)$ is a function satisfying some natural growth and dissipativity conditions, $h(x)$ is a deterministic external force acting on the system, and~$\eta$ is a random process, white in time and regular in the space variables; see Section~\ref{s2.3} for the exact hypotheses imposed on~$f$ and~$\eta$. The Cauchy problem~\eqref{1.1}--\eqref{1.3} is well posed in the space $H:=L^2(D,\R^k)$, and we denote by $\PPhi=\{\varphi_t:H\to H,t\ge0\}$ the corresponding RDS defined on a probability space $(\Omega,\FF,\IP)$ with a group of shift operators $\{\theta_t:\Omega\to\Omega,t\in\R\}$ (see Section~\ref{s2.3}). We have the following result on the existence of an exponential attractor for~$\PPhi$.

\begin{mta} 
There is a random compact set $\MM_\omega\subset H$ and an event $\Omega_*\subset\Omega$ of full measure such that the following properties hold for $\omega\in\Omega_*$. 

\smallskip
\noindent
{\sl Semi-invariance}. $\varphi_t^\omega(\MM_\omega)\subset \MM_{\theta_t\omega}$ for all $t\ge0$.

\smallskip
\noindent
{\sl Exponential attraction}.
There is $\beta>0$ such that for any ball $B\subset H$ we have
$$
\sup_{u\in B}\,\inf_{v\in \MM_{\theta_t\omega}}
\|\varphi_t^\omega(u)-v\|\le C(B)e^{-\beta t},
\quad t\ge0,
$$
where $C(B)$ is a constant depending only on~$B$.

\smallskip
\noindent
{\sl Finite-dimensionality}. 
There is a number $d>0$ such that $\dim_f(\MM_\omega)\le d$, where~$\dim_f$ stands for the fractal dimension of~$\MM_\omega$. 
\end{mta}

Note that this type of results are well known for non-autonomous dynamical systems (e.g., see~\cite{EMZ-2005,MZ-2008}). An essential difference between non-autonomous and stochastic systems is that the latter deal with forces which are, in general, unbounded in time, and some key quantities can be controlled only after taking the time average. This turns out to be sufficient for the construction of an exponential attractor.

\smallskip
Let us now assume that the random force~$\eta$ in Eq.~\eqref{1.1} is replaced by~$\e\eta$, where~$\e\in[-1,1]$ is a parameter. We denote by~$\MM_\omega^\e$ the corresponding exponential attractors. Since in the limit case $\e=0$ the equation is no longer stochastic, the corresponding attractor $\MM=\MM^0$ is also independent of~$\omega$.  A natural question is whether one can construct~$\MM_\omega^\e$ in such a way that the symmetric distance between the attractors of stochastic and deterministic equations goes to zero as $\e\to0$. The following theorem gives a positive answer to that question. 

\begin{mtb}
The exponential attractors~$\MM_\omega^\e$, $\e\in[-1,1]$, can be constructed in such a way that
$$
d^s(\MM_\omega^\e,\MM)\to0
\quad\mbox{almost surely as $\e\to0$},
$$
where $d^s$ stands for the symmetric distance between two subsets of~$H$. 
\end{mtb}

We refer the reader to Section~\ref{s5} for more precise statements of the results on the existence of exponential attractors and their dependence on a parameter. Let us note that various results similar to Theorem~B were established earlier in the case of deterministic PDE's; e.g., see the papers~\cite{FGMZ-2004,EMZ-2005}, the first of which is devoted to studying the behaviour of exponential attractors under singular perturbations, while the second deals with non-autonomous dynamical systems and proves H\"older continuous dependence of the exponential  attractor on a parameter.

We emphasize  that the convergence in Theorem~B differs from the one in the case of global attractors, for which, in general, only lower semicontinuity can be established. For instance, let us consider the following one-dimensional ODE perturbed by the time derivative of a standard Brownian motion~$w$:
\begin{equation} \label{toy}
\dot u=u-u^3+\e\dot w.
\end{equation}
When $\e=0$, the global attractor~$\aA$ for~\eqref{toy} is the interval~$[-1,1]$ and is regular in the sense that it consists of the stationary points and the unstable manifolds around them. It is well known that the regular structure of an attractor is very robust and survives rather general deterministic perturbations, and in many cases it is possible to prove that the symmetric distance between the attractors for the perturbed and unperturbed systems goes to zero; see~\cite{BV1992,CVZ-2012}. On the other hand, it is proved in~\cite{CF-1998} that the random attractor~$\aA_\omega^\e$ for~\eqref{toy} consists of a single trajectory and, hence, the symmetric distance between~$\aA$ and~$\aA_\omega^\e$ does not go to zero as~$\e\to0$.

In conclusion, let us mention that some results similar to those described above hold for other stochastic PDE's, including the 2D Navier--Stokes system. They will be considered in a subsequent publication. 

\medskip
The paper is organised as follows. In Section~\ref{s2}, we present some preliminaries on random dynamical systems and a reaction-diffusion equation perturbed by a spatially regular white noise. Section~\ref{s3} is devoted to some general results on the existence of exponential attractors and their dependence on a parameter. In Section~\ref{s5}, we apply our abstract construction to the stochastic reaction-diffusion system~\eqref{1.1}--\eqref{1.3}. Appendix gathers some results on coverings of random compact sets and their image under random mappings, as well as the time-regularity of stochastic processes.

\medskip
{\bf Acknowledgements}. This work was supported by the Royal Society--CNRS grant {\it Long time behavior of solutions for stochastic Navier--Stokes equations\/} (No.~YFDRN93583). The first author was supported by the ANR grant {\it STOSYMAP\/} (No.~ANR 2011 BS01 015 01). 

\subsection*{Notation}
Let $J\subset\R$ be an interval, let $D\subset\R^n$ be a bounded domain with smooth boundary~$\p D$, and let~$X$ be a Banach space. Given a compact set $\KK\subset X$, we denote by $\HH_\e(\KK,X)$ its Kolmogorov $\e$-entropy; see~\cite{lorentz1986}. If $Y$ is another Banach space with compact embedding $Y\Subset X$, then we write $\HH_\e(Y,X)$ for the $\e$-entropy of a unit ball in~$Y$ considered as a subset in~$X$. We denote by $\dot B_X(v,r)$ and~$B_X(v,r)$ the open and closed balls in~$X$ of radius~$r$ centred at~$v$ and by~$\OO_r(A)$ the closed $r$-neighbourhood of a subset $A\subset X$. The closure of~$A$ in~$X$ is denoted by~$[A]_X$. Given any set~$C$, we write~$\#C$ for the number of its elements. 

\medskip
We shall use the following function spaces:

\smallskip
\noindent
$L^p=L^p(D)$ denotes the usual Lebesgue space in~$D$ endowed with the standard  norm~$\|\cdot\|_{L^p}$. In the case $p=2$, we omit the subscript from the notation of the norm. We shall write $L^p(D,\R^k)$ if we need to emphasise the range of functions. 

\smallskip
\noindent
$W^{s,p}=W^{s,p}(D)$ stands for the standard Sobolev space with a norm~$\|\cdot\|_{s,p}$. In the case $p=2$, we write~$H^s=H^s(D)$ and~$\|\cdot\|_s$, respectively. We denote by $H_0^s=H^s_0(D)$ the closure in~$H^s$ of the space of infinitely smooth functions with compact support. 

\smallskip
\noindent
$C(J,X)$ stands for the space of continuous functions $f:J\to X$. 

\medskip
When describing a property involving a random parameter~$\omega$, we shall assume that it holds almost surely, unless specified otherwise. Given a random function $f_\omega:D\to X$, we shall say that it is (almost surely) H\"older-continuous if there is $\gamma\in(0,1)$ such that, for any bounded ball $B\subset \R^n$, we have
$$
\|f_\omega(t_1)-f_\omega(t_2)\|_X\le C_\omega|t_1-t_2|^\gamma,
\quad t_1,t_2\in B,
$$
where $C_\omega=C_\omega(B)$ is an almost surely finite random variable. If~$f$ depends on an additional parameter $y\in Y$ (that is, $f=f_\omega^y(t)$),  then we say that $f$ is H\"older-continuous uniformly in~$y$ if the above inequality holds for~$f_\omega^y(t)$ with a random constant~$C_\omega(B)$ not depending on~$y$.  

\smallskip
We denote by~$c_i$ and~$C_i$ unessential positive constants not depending on other parameters. 

\section{Preliminaries}
\label{s2}

\subsection{Random dynamical systems and their attractors}
\label{s2.1}
Let $(\Omega,\FF,\IP)$ be a complete probability space, $\{\theta_t,t\in\R\}$ be a group of measure-preserving transformations of~$\Omega$, and~$X$ be a separable Banach space. Recall that a {\it continuous random dynamical system in~$X$ over\/}~$\{\theta_t\}$ (or simply an {\it RDS in~$X$\/}) is defined as a family of continuous mappings $\PPhi=\{\varphi_t^\omega:X\to X,t\ge0\}$ that satisfy the following conditions:
\begin{description}
\item[Measurability.]
The mapping $(t,\omega,u)\mapsto \varphi_t^\omega(u)$ from $\R_+\times\Omega\times X$ to~$X$ is measurable with respect to the $\sigma$-algebras $\BB_{\R_+}\otimes\FF\otimes\BB_X$ and~$\BB_X$.
\item[Perfect co-cycle property.]
For almost every $\omega\in\Omega$, we have the identity
\begin{equation} \label{2.1}
\varphi_{t+s}^\omega=\varphi_t^{\theta_s\omega}\circ\varphi_s^\omega, \quad t,s\ge0.
\end{equation}
\item[Time regularity.]
For almost every $\omega\in\Omega$, the function $(t,\tau)\mapsto\varphi_t^{\theta_{\tau}\omega}(u)$, defined on $\R_+\times\R$ with range in~$X$, is H\"older-continuous with some deterministic exponent~$\gamma>0$, uniformly with respect to $u\in\KK$ for any compact subsets~$\KK\subset X$.
\end{description}
An example of RDS is given in the next subsection, which is devoted to some preliminaries on a reaction-diffusion system with a random perturbation. 

Large-time asymptotics of trajectories for RDS is often described in terms of attractors. This paper  deals with random exponential attractors, and we now define some basic concepts. 

Recall that the distance between a point $u\in X$ and a subset $F\subset X$ is given by $d(u,F)=\inf_{v\in F}\|u-v\|$. The Hausdorff and symmetric distances between two subsets is defined by
\begin{align*}
d(F_1,F_2)&=\sup_{u\in F_1}d(u,F_2),\\
d^s(F_1,F_2)&=\max\bigl\{d(F_1,F_2),d(F_2,F_1)\bigr\}.
\end{align*}
We shall write $d_X$ and~$d_X^s$ to emphasise that the distance is taken in the metric of~$X$. Let $\{\MM_\omega,\omega\in\Omega\}$ be a random compact set in~$X$, that is, a family of compact subsets such that the mapping $\omega\mapsto d(u,\MM_\omega)$ is measurable for any $u\in X$. 

\begin{definition} \label{d2.1}
A random compact set~$\{\MM_\omega\}$ is called a {\it random exponential attractor\/} for the RDS~$\{\varphi_t\}$ if there is a set of full measure $\Omega_*\in\FF$ such that the following properties hold for $\omega\in\Omega_*$.
\begin{description}
\item[Semi-invariance.] For any $t\ge0$, we have $\varphi_t^\omega(\MM_\omega)\subset\MM_{\theta_t\omega}$. 
\item[Exponential attraction.]
There is a constant~$\beta>0$  such that 
\begin{equation} \label{2.2}
d\bigl(\varphi_t^\omega(B),\MM_{\theta_t\omega}\bigr)
\le C(B) e^{-\beta t}
\quad\mbox{for $t\ge0$},
\end{equation}
where $B\subset H$ is an arbitrary ball and~$C(B)$ is a constant that depends only on~$B$.

\item[Finite-dimensionality.]
There is random variable $d_\omega\ge0$ which is finite on~$\Omega_*$ such that 
\begin{equation} \label{2.3}
\dim_f\bigl(\MM_\omega\bigr)\le d_\omega.
\end{equation}
\item[Time continuity.] 
The function $t\mapsto d^s\bigl(\MM_{\theta_t\omega},\MM_\omega\bigr)$ is H\"older-continuous on~$\R$ with some exponent~$\delta>0$. 
\end{description}
\end{definition}

We shall also need the concept of a {\it random absorbing set\/}. Recall that a random compact set $\aA_\omega$ is said to be {\it absorbing\/} for~$\PPhi$ if for any ball $B\subset X$ there is $T(B)\ge0$ such that
\begin{equation} \label{2.4}
\varphi_t^\omega(B)\subset \aA_{\theta_t\omega}
\quad\mbox{for $t\ge T(B)$, $\omega\in\Omega$}. 
\end{equation}
All the above definitions make sense also in the case of discrete time, that is, when the time variable varies on the integer lattice~$\Z$. The only difference is that the property of time continuity should be skipped for discrete-time RDS and their attractors. In what follows, we shall deal with both situations.

\subsection{Reaction-diffusion system perturbed by white noise}
\label{s2.3}
Let $D\subset \R^n$ be a bounded domain with a smooth boundary~$\p D$. We consider the reaction-diffusion system~\eqref{1.1}, \eqref{1.2}, in which 
$u=(u_1,\dots,u_k)^t$ is an unknown vector function and~$a$ is a $k\times k$ matrix such that 
\begin{equation} \label{2.5}
a+a^t>0.
\end{equation}
We assume that $f\in C^2(\R^k,\R^k)$ satisfies the following growth and dissipativity conditions: 
\begin{align}
\langle f(u),u\rangle&\ge -C+c|u|^{p+1},\label{2.6}\\ 
f'(u)+f'(u)^t&\ge -C I,\label{2.7}\\
|f'(u)|&\le C(1+|u|)^{p-1},\label{2.8}
\end{align}
where $\langle \cdot,\cdot\rangle$ stands for the scalar product in~$\R^k$, $f'(u)$ is the Jacobi matrix for~$f$, $I$~is the identity matrix, $c$~and~$C$ are positive constants, and $0\le p\le \frac{n+2}{n-2}$. As for the right-hand side of~\eqref{1.1}, we assume $h\in L^2(D,\R^k)$ is a deterministic function and $\eta$ is a {\it spatially regular white noise\/}. That is, 
\begin{equation} \label{2.9}
\eta(t,x)=\frac{\p}{\p t}\zeta(t,x), \quad \zeta(t,x)=\sum_{j=1}^\infty b_j\beta_j(t)e_j(x),
\end{equation}
where $\{\beta_j(t),t\in\R\}$ is a sequence of independent two-sided Brownian motions defined on a complete probability space~$(\Omega,\FF,\IP)$, $\{e_j\}$~is an orthonormal basis in~$L^2(D,\R^k)$ formed of the eigenfunctions of the Dirichlet Laplacian, and~$b_j$ are real numbers satisfying the condition
\begin{equation} \label{2.10}
\BBBB:=\sum_{j=1}^\infty b_j^2<\infty. 
\end{equation}
In what follows, we shall assume that~$(\Omega,\FF,\IP)$ is the canonical space; that is, $\Omega$ is the space of continuous functions $\omega:\R\to H$ vanishing at zero, $\IP$~is the law of~$\zeta$ (see~\eqref{2.9}), and~$\FF$ is the $\IP$-completion of the Borel $\sigma$-algebra. In this case, the process~$\zeta$ can be written in the form $\zeta^\omega(t)=\omega(t)$, and 
a group of shifts~$\theta_t$ acts on~$\Omega$ by the formula $(\theta_t\omega)(s)=\omega(t+s)-\omega(t)$. Furthermore, it is well known (e.g., see Chapter~VII in~\cite{stroock1993}) the restriction of~$\{\theta_t,t\in\R\}$ to any lattice~$T\Z$ is ergodic. 

Let us denote $H=L^2(D,\R^k)$ and $V=H_0^1(D,\R^k)$.
The following result on the well-posedness of problem~\eqref{1.1}--\eqref{1.3}  can be established by standard methods used in the theory of stochastic PDE's (e.g., see~\cite{DZ1992,flandoli-1994}). 

\begin{theorem} \label{t2.2} 
Under the above hypotheses, for any $u_0\in H$ there is a stochastic process $\{u(t),t\ge0\}$ that is adapted to the filtration generated by~$\zeta(t)$ and possesses the following properties:
\begin{description}
\item[Regularity:]
Almost every trajectory of~$u(t)$ belongs to the space 
$$
\XX=C(\R_+,H)\cap L_{\rm loc}^2(\R_+,V)\cap 
L_{\rm loc}^{p+1}(\R_+\times D).
$$
\item[Solution:]
With probability~$1$, we have the relation
$$
u(t)=u_0+\int_0^t\bigl(a\Delta u-f(u)+h\bigr)\,ds+\zeta(t), \quad t\ge0,
$$
where the equality holds in the space $H^{-1}(D)$.
\end{description}
Moreover, the process~$u(t)$ is unique in the sense that if~$v(t)$ is another process with the same properties, then with probability~$1$ we have $u(t)=v(t)$ for all~$t\ge0$.
\end{theorem}

The family of solutions for~\eqref{1.1}, \eqref{1.2} constructed in Theorem~\ref{t2.2} form an RDS in the space~$H$. Let us describe in more detail a set of full measure on which the perfect co-cycle property and the H\"older-continuity in time are true. 

\smallskip
Let us denote by~$z=z^\omega(t)$ the solution of the linear equation
\begin{equation} \label{2.24}
\dot z-a\Delta z=h+\eta(t),
\end{equation}
supplemented with the zero initial and boundary conditions. Such a solution exists and belongs to the space $\YY:=C(\R_+,H)\cap L_{\rm loc}^2(\R_+,V)$ with probability~$1$. Moreover, one can find a set~$\Omega_*\in\FF$ of full measure such that $\theta_t(\Omega_*)=\Omega_*$ for all $t\in\R$ and $z^\omega\in\YY$ for $\omega\in\Omega_*$. We now write a solution of~\eqref{1.1}--\eqref{1.3} in the form $u=z+v$ and note that~$v$ must satisfy the equation
\begin{equation} \label{2.25}
\dot v-a\Delta v+f(z+v)=0. 
\end{equation}
For any $\omega\in\Omega_*$ and $u_0\in H$, this equation has a unique solution $v\in\XX$ issued from~$u_0$. The RDS associated with~\eqref{1.1}--\eqref{1.2} can be written as 
$$
\varphi_t^\omega(u_0)=\left\{
\begin{array}{cl}
z^\omega(t)+v^\omega(t)&\mbox{for $\omega\in\Omega_*$},\\[2pt]
0&\mbox{for $\omega\notin\Omega_*$}. 
\end{array}
\right.
$$
Then $\PPhi=\{\varphi_t,t\ge0\}$ is an RDS in the sense defined in the beginning of Section~\ref{s2.1}, and the time continuity and perfect co-cycle properties hold on~$\Omega_*$.

\section{Abstract results on exponential attractors}
\label{s3}
\subsection{Exponential attractor for discrete-time RDS}
\label{s3.1}
Let $H$ be a Hilbert space and let $\PPsi=\{\psi_k^\omega,k\in\Z_+\}$ be a discrete-time RDS in~$H$ over a group of measure-preserving transformations $\{\sigma_k\}$ acting on a complete probability space $(\Omega,\FF,\IP)$. We shall assume that~$\PPsi$ satisfies the following condition. 

\smallskip
\noindent
\begin{condition} \label{H}
There is a Hilbert space~$V$ compactly embedded in~$H$, a random compact set~$\{\aA_\omega\}$, and constants~$m,r>0$ such that the properties below are satisfied. 
\begin{description}
\item[\it Absorption.]
The family~$\{\aA_\omega\}$ is a random absorbing set for~$\PPsi$.  
\item[\it Stability.]
With probability~$1$, we have
\begin{equation} \label{3.1}
\psi_1^\omega\bigl(\OO_r(\aA_\omega)\bigr)\subset \aA_{\sigma_1\omega}.
\end{equation}
\item[\it Lipschitz continuity.]
There is an almost surely finite random variable $K_\omega\ge1$ such that $K^m\in L^1(\Omega,\IP)$ and 
\begin{equation} \label{3.2}
\|\psi_1^\omega(u_1)-\psi_1^\omega(u_2)\|_V\le K_\omega \|u_1-u_2\|_H
\quad\mbox{for $u_1,u_2\in \OO_r(\aA_\omega)$}. 
\end{equation}
\item[\it Kolmogorov $\e$-entropy.]
There is a constant~$C$ and an almost surely finite random variable $C_\omega$ such that $C_\omega K_\omega^m\in L^1(\Omega,\IP)$, 
\begin{align}
\HH_\e(V,H)&\le C\,\e^{-m},\label{3.3}\\
\HH_\e(\aA_\omega,H)&\le C_\omega\e^{-m}.\label{3.4}
\end{align}
\end{description}
\end{condition}

The following theorem is an analogue for RDS of a well-known result on the existence of an exponential attractor for deterministic dynamical systems; e.g., see Section~3 of the paper~\cite{MZ-2008} and the references therein. 

\begin{theorem} \label{t3.1}
Assume that the discrete-time RDS~$\PPsi$ satisfies Condition~\ref{H}. Then~$\PPsi$ possesses an exponential attractor~$\MM_\omega$. Moreover, the attraction property holds for the norm of~$V$:
\begin{equation} \label{3.15}
d_V\bigl(\psi_k^\omega(B),\MM_{\sigma_k\omega}\bigr)
\le C(B) e^{-\beta k} \quad\mbox{for $k\ge0$},
\end{equation}
where $B\subset H$ is an arbitrary ball and~$C(B)$ and~$\beta>0$ are some constants not depending on~$k$. 
\end{theorem}

The proof given below will imply that~\eqref{3.15} holds for $B=\aA_\omega$ with $C(B)=r$, and that in inequality~\eqref{3.15} the constant in front of~$e^{-\beta k}$ has the form
\begin{equation} \label{rtime}
C(B)=2^{T(B)}r,
\end{equation}
where $T(B)$ is a time after which the image of the ball~$B$ under the mapping $\psi_k^\omega$ belongs  to the absorbing set~$\aA_{\sigma_t\omega}$. Furthermore, as is explained in Remark~\ref{r3.3} below, under an additional assumption, the fractal dimension~$\dim_f(\MM_\omega)$ can be bounded by a deterministic constant. 

\begin{proof}
We repeat the scheme used in the case of deterministic dynamical systems. However, an essential difference is that we have a random parameter and need to follow the dependence on it. In addition, the constants entering various inequalities are now (unbounded) random variables, and we shall need to apply the Birkhoff ergodic theorem to bound some key quantities. 

\medskip
{\it Step~1: An auxiliary construction}. 
Let us define a sequence of random finite sets~$V_k(\omega)$ in the following way. Applying Lemma~\ref{l6.1} with $\delta_\omega=(2K_\omega)^{-1}r$ to the random compact set~$\aA_{\omega}$, we construct a random finite set $U_0(\omega)$ such that 
\begin{gather} 
d^s\bigl(\aA_{\omega},U_0(\omega)\bigr)\le \delta_\omega,\label{3.5}\\
\ln\bigl(\#U_0(\omega)\bigr)\le 2^mC_\omega\delta_\omega^{-m}
\le(4/r)^mC_\omega K_{\omega}^m. 
\label{3.6}
\end{gather}
Since $K_\omega\ge1$, we have $\delta_\omega\le r/2$, whence it follows that $U_0(\omega)\subset\OO_r(\aA_{\omega})$. Setting $V_1(\sigma_1\omega)=\psi_1^{\omega}(U_0(\omega))$, in view of~\eqref{3.1}, \eqref{3.2}, and~\eqref{3.5}, we obtain
$$
\psi_1^{\omega}(\aA_{\omega})\subset
\bigcup_{u\in V_1(\sigma_1\omega)}B_V(u,r/2)=:\CC_1(\omega), \quad
V_1(\sigma_1\omega)\subset\OO_{r/2}\bigl(\psi_1^\omega(\aA_\omega)\bigr)\cap\aA_{\sigma_1\omega}.
$$
Now note that~$\CC_1(\omega)$ is a random compact set in~$H$. Moreover, it follows from~\eqref{3.3} and~\eqref{3.6} that
\begin{align}
\HH_\e(\CC_1(\omega),H)
&\le \ln\bigl(\#V_1(\sigma_1\omega)\bigr)+\HH_{2\e/r}(V,H)\notag\\
%\le 2^mC_\omega\Bigl(\frac{1}{\delta_\omega}\Bigr)^m+ C\Bigl(\frac{r}{2\e}\Bigr)^m\notag\\
&\le (4/r)^m C_\omega  K_{\omega}^m+(r/2)^mC\e^{-m}.
\label{3.7}
\end{align}
Applying Lemma~\ref{l6.1} with $\delta_\omega=(4K_{\sigma_{1}\omega})^{-1}r$ to~$\CC_1(\omega)$, we construct a random finite set $U_1(\omega)$ such that 
\begin{gather*} 
d^s\bigl(\CC_1(\omega),U_1(\omega)\bigr)\le \delta_\omega,\\
\ln\bigl(\#U_1(\omega)\bigr)\le \HH_{\delta_\omega/2}(\CC_1(\omega),H)
\le (4/r)^m C_\omega  K_{\omega}^m
+2^mC K_{\sigma_{1}\omega}^m. 
\end{gather*}
Repeating the above argument and setting $V_2(\sigma_2\omega)=\psi_1^{\sigma_{1}\omega}(U_1(\omega))$, we obtain
\begin{align*}
\psi_1^{\sigma_{1}\omega}\bigl(\CC_1(\omega)\bigr)&\subset
\bigcup_{u\in V_2(\sigma_2\omega)}B_V(u,r/4)=:\CC_2(\omega),\\
V_2(\sigma_2\omega)&\subset\OO_{r/4}\bigl(\psi_1^{\sigma_1\omega}(\CC_1(\omega)\bigr)\cap\aA_{\sigma_2\omega}.
\end{align*}
Moreover, $\CC_2(\omega)$ is a random compact set~$H$ whose $\e$-entropy satisfies the inequality (cf.~\eqref{3.7})
\begin{align*}
\HH_\e(\CC_2(\omega),H)
&\le \ln\bigl(\#V_2(\sigma_2\omega)\bigr)+\HH_{4\e/r}(V,H)\notag\\
&\le (4/r)^m C_\omega  K_{\omega}^m
+2^mC K_{\sigma_{1}\omega}^m+(r/4)^mC\e^{-m}.
\end{align*}
Iterating this procedure and recalling that $\sigma_k:\Omega\to\Omega$ is a one-to-one transformation, we construct random finite sets $V_k(\omega)$, $k\ge1$, and unions of balls
$$
\CC_k(\omega):=\bigcup_{u\in V_k(\sigma_k\omega)}B_V(u,2^{-k}r)
$$ 
such that the following properties hold for any integer~$k\ge1$: 
\begin{gather}
\psi_k^{\omega}(\aA_{\omega})\subset \CC_k(\omega),
\label{3.8}\\
V_k(\omega)\subset\OO_{2^{1-k}r}\bigl(\psi_1^{\sigma_{-1}\omega}(\CC_{k-1}(\sigma_{1-k}\omega))\bigr)\cap\aA_{\omega}, 
\label{3.13}\\
\ln\bigl(\#V_k(\omega)\bigr)
\le(4/r)^m C_{\sigma_{-k}\omega}  K_{\sigma_{-k}\omega}^m
+2^mC\sum_{j=1}^{k-1} K_{\sigma_{j-k}\omega}^m. 
\label{3.9}
\end{gather}

\smallskip
{\it Step~2: Description of an attractor}. 
Let us define a sequence of random finite sets by the rule
$$
E_1(\omega)=V_1(\omega), \qquad E_k(\omega)=V_k(\omega)\cup\psi_1^{\sigma_{-1}\omega}\bigl(E_{k-1}(\sigma_{-1}\omega)\bigr), \quad k\ge2. 
$$
The very definition of~$E_k$ implies that 
\begin{equation} \label{3.11}
\psi_1^\omega(E_k(\omega))\subset E_{k+1}(\sigma_1\omega).
\end{equation}
and since $\#V_{k}(\omega)\le \#V_{k+1}(\sigma_1\omega)$, it follows from~\eqref{3.9} that 
\begin{align}
\ln\bigl(\#E_k(\omega)\bigr)
&\le \ln k+\ln\bigl(\#V_{k}(\omega)\bigr)\notag\\
&\le \ln k+(4/r)^m C_{\sigma_{-k}\omega}  K_{\sigma_{-k}\omega}^m
+2^mC\sum_{j=1}^{k-1} K_{\sigma_{j-k}\omega}^m. 
\label{3.10}
\end{align}
Furthermore, it follows from~\eqref{3.8} that
\begin{equation} \label{3.12}
d_V\bigl(\psi_k^\omega(\aA_\omega),V_k(\sigma_k\omega)\bigr)\le 2^{-k}r, \quad k\ge0.
\end{equation}
We now define a random compact set~$\MM_\omega$ by the formulas
\begin{equation} \label{3.016}
\MM_\omega=\bigl[\MM_\omega'\bigr]_V\,, \quad
\MM_\omega'=\bigcup_{k=1}^\infty E_k(\omega). 
\end{equation}
We claim that~$\MM_\omega$ is a random exponential attractor for~$\PPsi$. Indeed, the semi-invariance follows immediately from~\eqref{3.11}. Furthermore, inequality~\eqref{3.12} implies that 
$$
d_V\bigl(\psi_k^\omega(\aA_\omega),\MM_{\sigma_k\omega}\bigr)\le 2^{-k}r
\quad\mbox{for any $k\ge0$}. 
$$
Recalling that~$\aA_\omega$ is an absorbing set and using inclusion~\eqref{2.4}, together with the co-cycle property, we obtain
$$
d_V\bigl(\psi_k^\omega(B),\MM_{\sigma_k\omega}\bigr)
\le d_V\bigl(\psi_{k-T}^{\sigma_T\omega}(\aA_{\sigma_T\omega}), \MM_{\sigma_{k-T}(\sigma_T\omega)}\bigr)\le 2^{T-k}r,
$$
where $T=T(B)$ is the constant entering~\eqref{2.4}. This implies the exponential attraction inequality~\eqref{3.15} with $\beta=\ln 2$ and $C(B)=2^{T(B)}r$. It remains to prove that~$\MM_\omega$ has a finite  fractal dimension. This is done in the next step. 

\smallskip
{\it Step~3: Estimation of the fractal dimension}. We shall need the following lemma, whose proof is given at the end of this subsection. 
% For two integers $a\le b$, we write $[\![a,b]\!]=[a,b]\cap\Z$. 

\begin{lemma} \label{l3.2}
Under the hypotheses of Theorem~\ref{t3.1}, for any integers $l\ge0$, $k\in\Z$, and $m\in[0,l]$, we have 
\begin{equation} \label{3.14}
d_V\bigl(E_k(\sigma_k\omega),\psi_m^{\sigma_{k-m}\omega}(\aA_{\sigma_{k-m}\omega})\bigr)\le 2^{2-(k-l)}r\prod_{j=1}^{l}K_{\sigma_{k-j}\omega}.
\end{equation}
\end{lemma}

Inequality~\eqref{3.14} with $m=l$ and~$\omega$ replaced by~$\sigma_{-k}\omega$ implies that 
\begin{equation} \label{3.16}
d_V\bigl(E_k(\omega),\psi_{l}^{\sigma_{-l}\omega}(\aA_{\sigma_{-l}\omega})\bigr)\le 2^{2-(k-l)}r\prod_{j=1}^{l}K_{\sigma_{-j}\omega},
\end{equation}
where $k\ge1$ is arbitrary. On the other hand, in view of~\eqref{3.12} with $k=l$ and $\omega$ replaced by~$\sigma_{-l}\omega$, we have
$$
d_V\bigl(\psi_{l}^{\sigma_{-l}\omega}(\aA_{\sigma_{-l}\omega}),V_l(\omega)\bigr)\le 2^{-l}r. 
$$
Combining this with~\eqref{3.16}, we obtain
\begin{equation} \label{3.17}
d_V\biggl(\,\bigcup_{k\ge n}E_k(\omega),V_l(\omega)\biggr)
\le r\Bigl(2^{-l}+2^{2-(n-l)}\prod_{j=1}^{l}K_{\sigma_{-j}\omega}\Bigr),
\end{equation}
where $n\ge1$ and $l\in[1,n]$ are arbitrary integers. Since $\{E_k(\omega)\}$ is an increasing sequence and $V_l(\omega)\subset E_l(\omega)\subset\MM_\omega$ for any $l\ge1$, inequality~\eqref{3.17} implies that
\begin{equation} \label{3.18}
d_V^s\biggl(\MM_\omega,\bigcup_{l=1}^n V_l(\omega)\biggr)
\le r\inf_{l\in[1,n]} 
\Bigl(2^{-l}+2^{2-(n-l)}\prod_{j=1}^{l}K_{\sigma_{-j}\omega}\Bigr)=:\e_n(\omega),
\end{equation}
where $n\ge1$ is arbitrary. If we denote by $N_\e(\omega)$ the minimal number of balls of radius~$\e>0$ that are needed to cover~$\MM_\omega$, then inequality~\eqref{3.18} implies that $N_{\e_n(\omega)}(\omega)\le \sum_{k=1}^n\#V_k(\omega)$. Since $V_k\subset E_k$ and $\#V_k(\omega)\ge \#V_{k-1}(\sigma_{-1}\omega)$, it follows from~\eqref{3.9} that 
\begin{align}
\ln N_{\e_n(\omega)}(\omega)
&\le \ln \bigl(n\,\#E_n(\omega)\bigr)\le\ln \bigl(n^2\,\#V_n(\omega)\bigr)\notag\\
&\le 2\ln n+(4/r)^m C_{\sigma_{-n}\omega} K_{\sigma_{-n}\omega}^m
+ 2^m C\sum_{k=1}^{n-1} K_{\sigma_{-k}\omega}^m.\label{3.19}
\end{align}
Since $K^m\in L^1(\Omega,\IP)$, by the Birkhoff ergodic theorem (see Section~1.6 in~\cite{walters1982}), we have
\begin{equation} \label{3.20}
\lim_{n\to\infty}\frac1n\sum_{k=1}^n K_{\sigma_{-k}\omega}^m= \xi_\omega,
\end{equation}
where $\xi_\omega$ is an integrable random variable. This implies, in particular, that $n^{-1} K_{\sigma_{-n}\omega}^m\to0$ as $n\to\infty$. By a similar argument,  $n^{-1} C_{\sigma_{-n}\omega}K_{\sigma_{-n}\omega}^m\to0$ as $n\to\infty$. Combining this with~\eqref{3.20} and~\eqref{3.19}, we derive
\begin{equation} \label{3.21}
\ln N_{\e_n(\omega)}(\omega)\le 2^mC \xi_\omega n+o_\omega(n),
\end{equation}
where, given $\alpha\in\R$, we denote by~$o_\omega(n^\alpha)$ any sequences of positive random variables such that $n^{-\alpha}o_\omega(n^\alpha)\to0$ a.\,s.\ as $n\to\infty$. On the other hand, since the function~$\log_2 x$ is concave, it follows from~\eqref{3.20} that 
\begin{equation} \label{3.024}
\frac{m}{l}\sum_{j=1}^l\log_2K_{\sigma_{-j}\omega}\le 
\log_2\Bigl(\frac1l\sum_{j=1}^l K_{\sigma_{-j}\omega}^m\Bigr)
=\log_2\bigl(\xi_\omega+o_\omega(1)\bigr),
\end{equation}
whence we conclude that the random variable~$\e_n$ defined in~\eqref{3.18} satisfies the inequality
$$
\e_n(\omega)\le r\inf_{l\in[1,n]}
\bigl(2^{-l}+4\cdot 2^{-(n-l)}(\xi_\omega+o_\omega(1))^{l/m}\bigr). 
$$
Taking $l=mn\bigl(2m+\log_2(\xi_\omega+o_\omega(1))\bigr)^{-1}$, we obtain
\begin{equation} \label{eps}
\e_n(\omega)
\le 5\exp\Bigl(-\frac{mn\ln2}{2m+\log_2(\xi_\omega+o_\omega(1))}\Bigr). 
\end{equation}
Combining this inequality with~\eqref{3.21}, we derive
$$
\lim_{n\to\infty}\frac{\ln N_{\e_n(\omega)}(\omega)}{\ln\e_n^{-1}(\omega)}
\le \frac{2^mC\xi_\omega(\ln\xi_\omega+2m)}{m\ln 2}=:d_\omega.
$$
It is now straightforward to see 
\begin{equation} \label{3.22}
\dim_f(\MM_\omega)=\limsup_{\e\to0^+}\frac{\ln N_\e}{\ln\e^{-1}}
\le d_\omega.
\end{equation}
The proof of the theorem is complete. 
\end{proof}

\begin{remark} \label{r3.3}
It follows from~\eqref{3.22} that if the random variable~$\xi_\omega$ entering the Birkhoff theorem is bounded (see~\eqref{3.20}), then the fractal dimension of~$\MM_\omega$ can be bounded by a deterministic constant. For instance, if the group of shift operators~$\{\sigma_k\}$ is ergodic, then~$\xi_\omega$ is constant, and the conclusion holds. This observation will be important in applications of Theorem~\ref{t3.1}. 
\end{remark}

\begin{proof}[Proof of Lemma~\ref{l3.2}]
The co-cycle property~\eqref{2.1} and inclusion~\eqref{3.1} imply that $\psi_m^{\sigma_{k-m}\omega}(\aA_{\sigma_{k-m}\omega})\supset \psi_l^{\sigma_{k-l}\omega}(\aA_{\sigma_{k-l}\omega})$ for $m\le l$. Hence, it suffices to establish~\eqref{3.14} for $m=l$.

We first note that inequality~\eqref{3.2}, inclusion~\eqref{3.13}, and the definition of~$\CC_k(\omega)$ imply that\,\footnote{In the case $n=1$, the left-hand side of this inequality is zero.}
$$
d_V\bigl(V_n(\omega),\psi_1^{\sigma_{-1}\omega}(V_{n-1}(\sigma_{-1}\omega)\bigr)\le 2^{2-n}rK_{\sigma_{-1}\omega},
$$
where $n\ge1$ is an arbitrary integer, and we set $V_0(\omega)=U_0(\omega)$. Combining this with~\eqref{3.2} and the co-cycle property, for any integers $n\ge1$ and $q\ge0$ we  derive
\begin{equation} \label{3.25}
d_V\bigl(\psi_q^\omega(V_n(\omega)),\psi_{q+1}^{\sigma_{-1}\omega}(V_{n-1}(\sigma_{-1}\omega)\bigr)
\le 2^{2-n}r\prod_{j=0}^qK_{\sigma_{j-1}\omega}.
\end{equation}
Applying~\eqref{3.25} to the pairs $(n,q)=(k-i,i)$, $i=0,\dots,l-1$, with~$\omega$ replaced by~$\sigma_{k-i}\omega$, using the triangle inequality, and recalling that $K_\omega\ge1$, we obtain
\begin{align*} 
d_V\bigl(V_k(\sigma_k\omega),\psi_{l}^{\sigma_{k-l}\omega}(V_{k-l}(\sigma_{k-l}\omega))\bigr)
&\le r\sum_{i=0}^{l-1} 2^{2-k+i}\prod_{j=0}^i K_{\sigma_{k-i+j-1}\omega}\\
&\le 2^{2-(k-l)} r\prod_{j=1}^{l}K_{\sigma_{k-j}\omega}\,,
\end{align*}
where $k\ge1$ and $p\in[1,k]$ are arbitrary integers. A similar argument based on the application of~\eqref{3.25} to the pairs $(n,q)=(k-s-i,s+i)$, $i=0,\dots, l-s-1$, with~$\omega$ replaced by~$\sigma_{k-s}\omega$, enables one to prove that for any integer $n\in[1,k]$ we have 
\begin{equation}
d_V\bigl(\psi_s^{\sigma_{k-s}\omega}(V_{k-s}(\sigma_{k-s}\omega)),\psi_{l}^{\sigma_{k-l}\omega}(V_{k-l}(\sigma_{k-l}\omega))\bigr)
\le2^{2-(k-l)} r\prod_{j=1}^{l}K_{\sigma_{k-j}\omega}\,,
\label{3.26}
\end{equation}
where $s\in[0,l-1]$ is an arbitrary integer. Recalling that $V_n(\omega)\subset\aA_\omega$ for any $n\ge1$ (see~\eqref{3.13}), we deduce from~\eqref{3.26} that
\begin{equation}
d_V\bigl(\psi_s^{\sigma_{k-s}\omega}(V_{k-s}(\sigma_{k-s}\omega)),\psi_{l}^{\sigma_{k-l}\omega}(\aA_{\sigma_{k-l}\omega})\bigr)
\le2^{2-(k-l)} r\prod_{j=1}^{l}K_{\sigma_{k-j}\omega}
\label{3.27}
\end{equation}
for any integer $s\in[0,k-1]$. Since 
$$
E_k(\sigma_k\omega)
=\bigcup_{s=0}^{k-1}
\psi_{s}^{\sigma_{k-s}\omega}(V_{k-s}(\sigma_{k-s}\omega)),
$$
inequality~\eqref{3.27} immediately implies~\eqref{3.14} with $m=l$. 
\end{proof}

\subsection{Dependence of attractors on a parameter}
\label{s3.2}
We now turn to the case in which the RDS in question depends on a parameter. Namely, let $Y\subset\R$ and $\TT\subset\R$ be bounded closed intervals. We consider a discrete-time RDS $\PPsi^y=\{\psi_k^{y,\omega}:H\to H, k\ge0\}$ depending on the parameter $y\in Y$ and a family of measurable isomorphisms $\{\theta_\tau:\Omega\to\Omega,\tau\in\TT\}$. We assume that~$\theta_\tau$ commutes with~$\sigma_1$ for any $\tau\in\TT$, and the following uniform version of Condition~\ref{H} is satisfied. 

\smallskip
\noindent
\begin{condition}\label{UH} 
There is a Hilbert space~$V$ compactly embedded in~$H$, almost surely finite random variables $R_\omega^y, R_\omega\ge0$, and positive constants~$m$, $r$, and $\alpha\le1$ such that $R_\omega^y\le R_\omega$ for all $y\in Y$, and the  following properties hold.
\begin{description}
\item[\it Absorption and continuity.]
For any ball $B\subset H$ there is a time $T(B)\ge0$ such that  
\begin{equation} \label{3.027}
\psi_k^{y,\theta_\tau\omega}(B)\subset \aA_\omega^y\quad \mbox{for $k\ge T(B)$, $y\in Y$, $\tau\in\TT$, $\omega\in\Omega$},
\end{equation}
where we set $\aA_\omega^y=B_V(R_\omega^y)$. 
Moreover, there is an integrable  random variable~$L_\omega\ge1$  such that
\begin{equation} \label{3.28}
|R_{\theta_{\tau_1}\omega}^{y_1}-R_{\theta_{\tau_2}\omega}^{y_2}|
\le L_\omega\bigl(|y_1-y_2|^\alpha+|\tau_1-\tau_2|^\alpha\bigr)
%d^s(\aA_{\theta_{\tau_1}\omega}^{y_1},\aA_{\theta_{\tau_1}\omega}^{y_2})\le L_\omega \bigl(|y_1-y_2|+|\tau_1-\tau_2|^\alpha\bigr)
\end{equation}
for $y_1,y_2\in Y$, $\tau_1,\tau_2\in\TT$, and $\omega\in\Omega$.
\item[\it Stability.]
With probability~$1$, we have
\begin{equation} \label{3.29}
\psi_1^{y,\omega}\bigl(\OO_r(\aA_\omega^y)\bigr)
\subset \aA_{\sigma_1\omega}^y\quad\mbox{for $y\in Y$}.
\end{equation}
\item[\it H\"older continuity.]
There are almost surely finite random variables $K_\omega^y,K_\omega\ge1$ such that $K_\omega^y\le K_\omega$ for all $y\in Y$, $(RK)^m\in L^1(\Omega,\IP)$, and \begin{equation} \label{3.30}
\|\psi_1^{y_1,\theta_{\tau_1}\omega}(u_1)-\psi_1^{y_2,\theta_{\tau_2}\omega}(u_2)\|_V
\le K_\omega^{y_1,y_2} \bigl(|y_1-y_2|^\alpha+|\tau_1-\tau_2|^\alpha+\|u_1-u_2\|_H\bigr) 
\end{equation}
for $y_1,y_2\in Y$, $\tau_1,\tau_2\in\TT$, $u_1,u_2\in \OO_r(\aA_\omega^{y_1}\cup \aA_\omega^{y_2})$, and $\omega\in\Omega$, where we set $K_\omega^{y_1,y_2}=\max(K_\omega^{y_1},K_\omega^{y_2})$.
\item[\it Kolmogorov $\e$-entropy.]
Inequalities~\eqref{3.3} holds with some~$C$ not depending on~$\e$. 
\end{description}
\end{condition}

In particular, for any fixed $y\in Y$, the RDS $\PPsi^y$ satisfies Condition~\ref{H} and, hence, possesses an exponential attractor~$\MM_\omega^y$. The following result is a refinement of Theorem~\ref{t3.1}.

\begin{theorem} \label{t3.3}
Let  $\PPsi^y$ be a family of RDS satisfying Condition~\ref{UH}. Then there is a random compact set $(y,\omega)\mapsto\MM_\omega^y$ with the underlying space~$Y\times \Omega$ and a set of full measure $\Omega_*\in\FF$ such that the following properties hold. 

\smallskip
{\bf Attraction.} 
For any $y\in Y$, the family~$\{\MM_\omega^y\}$ is a random exponential attractor for $\PPsi^y$. Moreover, the attraction property holds uniformly in~$y$ and~$\omega$ in the following sense: for any ball $B\subset H$ there is $C(B)>0$ such that
\begin{equation} \label{3.033}
\sup_{y\in Y}d_V\bigl(\psi_k^{y,\omega}(B),\MM_{\sigma_k\omega}^y\bigr)
\le C(B) e^{-\beta k} \quad\mbox{for $k\ge0$, $\omega\in\Omega_*$},
\end{equation}
where $\beta>0$ is a constant not depending on~$B$, $k$, $y$, and~$\omega$. 

\smallskip
{\bf H\"older continuity.}
There are finite random variables~$P_\omega$ and $\gamma_\omega\in(0,1]$ such that
\begin{equation} \label{3.31}
d_V^s(\MM_{\omega}^{y_1},\MM_{\omega}^{y_2})
\le P_\omega |y_1-y_2|^{\gamma_\omega}
\quad\mbox{for $y_1,y_2\in Y$, $\omega\in\Omega_*$}.
\end{equation}
If, in addition, the random variable~$\xi_\omega$ entering~\eqref{3.20} is bounded, then~$\gamma_\omega$ can be chosen to be constant, and we have the inequality
\begin{equation} \label{3.031}
d_V^s(\MM_{\theta_{\tau_1}\omega}^{y},\MM_{\theta_{\tau_2}\omega}^{y})
\le Q_\omega |\tau_1-\tau_2|^\gamma\quad
\mbox{for $y\in Y$, $\tau_1,\tau_2\in\TT$, $\omega\in\Omega_*$},
\end{equation}
where $\gamma\in(0,1]$, and $Q_\omega$ is a finite random constant. 
\end{theorem}

In addition, it can be shown that all the moments of the random variables~$P_\omega$ and~$Q_\omega$ are finite. The proof of this property requires some estimates for the rate of convergence in the Birkhoff ergodic theorem. Those estimates can be derived from exponential bounds for the time averages of some  norms of solutions. Since the corresponding argument is technically rather complicated, we shall confine ourselves to the proof of the result stated above.
\begin{proof}[Proof of Theorem~\ref{t3.3}]
To establish the first assertion, we repeat the scheme used in the proof of Theorem~\ref{t3.1}, applying Corollary~\ref{c5.3} and Lemma~\ref{l5.4} instead of Lemma~\ref{l6.1} to construct coverings of random compact sets. Namely, let us denote by 
$U_k^y(\omega)$, $V_k^y(\omega)$, and $\CC_k^y(\omega)$ with $y\in Y$ the random sets described in the proof of Theorem~\ref{t3.1} for the RDS~$\PPsi^y$. In particular, $U_k^y(\omega)$ is a random finite set such that
\begin{equation} \label{3.32}
d^s\bigl(U_k^y(\omega),\CC_k^y(\omega)\bigr)
\le \bigl(2^{k+1}K_{\sigma_k\omega}\bigr)^{-1}r,\quad k\ge0,
\end{equation}
where $\CC_0^y(\omega)=\aA_\omega^y$ and 
\begin{equation} \label{3.33}
\CC_k^y(\omega)=\bigcup_{u\in V_k^y(\sigma_k\omega)}B_V(u,2^{-k}r),\quad V_k^y(\sigma_k\omega)=\psi_1^{y,\sigma_{k-1}\omega}
\bigl(U_{k-1}^y(\omega)\bigr)
\end{equation}
for $k\ge1$. We apply Corollary~\ref{c5.3} to construct a random finite set $R\mapsto U_{0,R}$ satisfying \eqref{5.014}--\eqref{5.016} with $\delta=\frac{r}{2K_\omega^y}$ and then define $U_0^y(\omega):=U_{0,R_\omega^y}$. The subsequent sets~$U_k^y(\omega)$, $k\ge1$, are constructed with the help of Lemma~\ref{l5.4}. What has been said implies the following bound for the number of elements of~$U_k^y(\omega)$ (cf.~\eqref{3.9}): 
$$
\ln\bigl(\#U_k^y(\omega)\bigr)
\le 4\bigl(\tfrac{32}{r}\bigr)^mC_\omega^y K_\omega^m
+4^m C\sum_{j=1}^kK_{\sigma_j\omega}^m,
$$
where $C_\omega^y=C(R_\omega^y)^m$. 
This enables one to repeat the argument of the proof of Theorem~\ref{t3.1} and to conclude that the random compact set defined by relations~\eqref{3.016} is an exponential attractor for~$\PPsi^y$ (with a uniform rate of attraction). 

\medskip
We now turn to the property of H\"older continuity for~$\MM_\omega^y$. Inequalities~\eqref{3.31} and~\eqref{3.031} are proved by similar arguments, and therefore we give a detailed proof for the first of them and confine ourselves to the scheme of the proof for the other. Inequality~\eqref{3.31} is established in four steps.

\smallskip
{\it Step~1}. 
We first show that
\begin{equation} \label{3.35}
d_V^s\bigl(V_k^{y_1}(\omega),V_k^{y_2}(\omega)\bigr)
\le|y_1-y_2|^\alpha\sum_{j=1}^{k}\prod_{i=1}^{j}K_{\sigma_{-i}\omega},
\end{equation}
where $|y_1-y_2|\le1$. The proof is by induction on~$k$. For $k=1$, the random finite set $U_0^y(\omega)$ does not depend on~$\omega$. Recalling that $V_1^y(\omega)=\psi_1^{y,{\sigma_{-1}\omega}}(U_0^y(\sigma_{\sigma_{-1}}\omega))$ and using~\eqref{3.30}, for $|y_1-y_2|\le1$ we get the inequality
$$
d_V^s\bigl(V_1^{y_1}(\omega),V_1^{y_2}(\omega)\bigr)
\le K_{\sigma_{-1}\omega} |y_1-y_2|^\alpha,
$$
which coincides with~\eqref{3.35} for $k=1$. Assuming that inequality~\eqref{3.35} is established for $1\le k\le m$, let us prove it for~$k=m+1$.  In view of Lemma~\ref{l5.4}, the random finite set~$U_m^y(\omega)$  satisfying~\eqref{3.32} can be constructed in such a way that
$$
d^s\bigl(U_m^{y_1}(\omega),U_m^{y_2}(\omega)\bigr)
\le d^s\bigl(V_m^{y_1}(\sigma_m\omega),V_m^{y_2}(\sigma_m\omega)\bigr).
$$
Combining this with~\eqref{3.30}, we see that
$$
d_V^s\bigl(V_{m+1}^{y_1}(\omega),V_{m+1}^{y_2}(\omega)\bigr)
\le K_{\sigma_{-1}\omega}\bigl\{|y_1-y_2|^\alpha+d_V^s\bigl(V_m^{y_1}(\sigma_{-1}\omega),V_m^{y_2}(\sigma_{-1}\omega)\bigr)\bigr\}. 
$$
Using inequality~\eqref{3.35} with $k=m$ and~$\omega$ replaced by~$\sigma_{-1}\omega$ to estimate the second term on the right-hand side, we arrive at~\eqref{3.35} with $k=m+1$. 

\smallskip
{\it Step~2}. 
We now prove that
\begin{equation} \label{3.36}
d_V^s(\MM_\omega^{y_1},\MM_\omega^{y_2})
\le 2\e_n(\omega)
+|y_1-y_2|^\alpha \sum_{k=1}^nk\prod_{i=1}^k K_{\sigma_{-i}\omega},
\end{equation}
where $n\ge1$ is an arbitrary integer and~$\e_n(\omega)$ is defined in~\eqref{3.18}. Indeed, inequality~\eqref{3.18}, which was proved in the case of a single RDS, remains true in the present parameter-dependent setting:
$$
d_V^s\biggl(\MM_\omega^{y_p},\bigcup_{k=1}^n V_k^{y_p}(\omega)\biggr)\le\e_n(\omega), \quad p=1,2.
$$
Combining this with~\eqref{3.36} and the obvious inequality 
$$
d_V^s(A_1\cup A_2,B_1\cup B_2)\le d_V^s(A_1,B_1)+d_V^s(A_2,B_2),
$$ 
we derive
$$
d_V^s(\MM_\omega^{y_1},\MM_\omega^{y_2})
\le 2\e_n(\omega)
+\sum_{k=1}^n d_V^s(V_k^{y_1}(\omega),V_k^{y_2}(\omega)).
$$
Using~\eqref{3.35} to estimate each term of the sum on the right-hand side, we arrive at~\eqref{3.36}.

\smallskip
{\it Step~3}. 
Suppose now we have shown that 
\begin{equation} \label{3.38}
\sum_{k=1}^nk\prod_{i=1}^k K_{\sigma_{-i}\omega}\le\exp(\zeta_\omega^n\,n), \quad n\ge1,
\end{equation}
where $\zeta_\omega^n\ge1$ is a sequence of almost surely finite random variables such that 
$$
\lim_{n\to\infty}\zeta_\omega^n=:\zeta_\omega<\infty
\quad\mbox{with probability~$1$}. 
$$
In this case, combining~\eqref{3.36} with~\eqref{eps} and~\eqref{3.38}, we derive  
\begin{equation} \label{3.39}
d_V^s(\MM_\omega^{y_1},\MM_\omega^{y_2})
\le 10\exp(-\eta_\omega^n\,n)+\exp(\zeta_\omega^n\,n)|y_1-y_2|^\alpha, 
\end{equation}
where we set 
$$
\eta_\omega^n=\frac{m\ln 2}{2m+\log_2(\xi_\omega+o_\omega(1))}.
$$
We wish to optimize the choice of~$n$ in~\eqref{3.39}. To this end, first note that
\begin{equation*} 
\lim_{n\to\infty}\eta_\omega^n=:\eta_\omega
=\frac{m\ln 2}{2m+\log_2\xi_\omega}>0.
\end{equation*}
Let $n_1(\omega)\ge1$ be the smallest integer such that
\begin{equation} \label{3.40}
\zeta_\omega^n\le 2\zeta_\omega, \quad 
\eta_\omega^n\ge \frac{\eta_\omega}{2}
\quad\mbox{for $n\ge n_1(\omega)$},
\end{equation}
and let $n_2(\omega,r)$ be the smallest integer greater than $2\eta_\omega^{-1}\gamma_\omega\ln r^{-1}$, where the small random constant $\gamma_\omega>0$ will be chosen later. Note that if
$$
r\le \tau_\omega:=\exp\bigl(-\tfrac{n_1(\omega)\eta_\omega}{2\gamma_\omega}\bigr),
$$
then $n_2(\omega,r)\ge n_1(\omega)$. Combining~\eqref{3.39} and~\eqref{3.40}, for $|y_1-y_2|\le\tau_\omega$ and $n=n_2(\omega,|y_1-y_2|)$,  we obtain
\begin{align*}
d_V^s(\MM_\omega^{y_1},\MM_\omega^{y_2})
&\le 10\exp(-\eta_\omega n/2)+\exp(2\zeta_\omega n)|y_1-y_2|^\alpha\\
&\le 10\,|y_1-y_2|^{\gamma_\omega}
+|y_1-y_2|^{\alpha-8\zeta_\omega\gamma_\omega/\eta_\omega}. 
\end{align*}
Choosing 
\begin{equation} \label{3.41}
\gamma_\omega=\frac{\alpha \eta_\omega}{\eta_\omega+8\zeta_\omega},
\end{equation}
we obtain 
$$
d_V^s(\MM_\omega^{y_1},\MM_\omega^{y_2})
\le 11\,|y_1-y_2|^{\gamma_\omega}
\quad\mbox{for $|y_1-y_2|\le\tau_\omega$}.
$$
This obviously implies the required inequality~\eqref{3.31} with an almost surely finite random constant~$P_\omega$.  

\smallskip
{\it Step~4}. It remains to prove~\eqref{3.38}. In view of~\eqref{3.024}, we have
\begin{equation} \label{3.42}
\prod_{i=1}^k K_{\sigma_{-i}\omega}
\le\bigl(\xi_\omega+o_\omega(1)\bigr)^{n/m}
\quad\mbox{for $1\le k\le n$}. 
\end{equation}
It follows that 
$$
\sum_{k=1}^nk\prod_{i=1}^k K_{\sigma_{-i}\omega}
\le \frac{n(n+1)}{2}\bigl(\xi_\omega+o_\omega(1)\bigr)^{n/m},
$$
whence we obtain~\eqref{3.38} with 
\begin{equation} \label{3.45}
\zeta_\omega^n=\zeta_\omega+o_\omega(1), \quad 
\zeta_\omega=\frac1m\ln\xi_\omega. 
\end{equation}
This completes the proof of~\eqref{3.31}. It is straightforward to see from~\eqref{3.41} and the explicit formulas for~$\zeta_\omega$ and~$\eta_\omega$ that if~$\xi_\omega\ge1$ is bounded, then~$\gamma_\omega$ can be chosen to be independent of~$\omega$. 

\medskip
We now turn to the scheme of the proof of~\eqref{3.031}. Suppose we have shown that (cf.~\eqref{3.35})
\begin{equation} \label{3.47}
d_V^s\bigl(V_k^{y}(\theta_{\tau_1}\omega),V_k^{y}(\theta_{\tau_1}\omega)\bigr)\le D_k(\omega)|\tau_1-\tau_2|^\alpha\quad
\mbox{for $y\in Y$, $|\tau_1-\tau_2|\le1$}, 
\end{equation}
where we set
\begin{equation} \label{3.48}
D_k(\omega)=c\,L_{\sigma_{-k}\omega}
\prod_{i=1}^{k}K_{\sigma_{-i}\omega}
+\sum_{j=1}^{k}\prod_{i=1}^{j}K_{\sigma_{-i}\omega},
\end{equation}
and $c\ge1$ is the constant in~\eqref{6.10}. In this case, repeating the argument used in Step~2, we derive (cf.~\eqref{3.36})
\begin{equation} \label{3.49}
d_V^s(\MM_{\theta_{\tau_1}\omega}^{y},\MM_{\theta_{\tau_2}\omega}^{y})
\le \e_n(\theta_{\tau_1}\omega)+\e_n(\theta_{\tau_2}\omega)
+D_n(\omega)\,|\tau_1-\tau_2|^\alpha,
\end{equation}
where $n\ge1$ is an arbitrary integer and~$\e_n(\omega)$ is defined in~\eqref{3.18}. If we prove that (cf.~\eqref{3.38})
\begin{equation} \label{3.50}
D_n(\omega)\le \exp(\zeta_\omega^n\,n), \qquad \lim_{n\to\infty}\zeta_\omega^n=\zeta\in\R_+\quad\mbox{a.\,s.},
\end{equation}
then the argument of Step~3 combined with the boundedness of~$\xi_\omega$ implies the required inequality~\eqref{3.031}. To prove~\eqref{3.50}, note that, by the Birkhoff theorem, there is an integrable random variable $\lambda_\omega\ge1$ such that
$$ %\begin{equation} \label{3.44}
\sum_{k=1}^n L_{\sigma_{-k}\omega}=n\lambda_\omega+o_\omega(n), \quad n\ge1.
$$ %\end{equation}
Combining this with~\eqref{3.38}, we obtain inequality~\eqref{3.50} (with  larger random variables~$\zeta_\omega^n$). 

\smallskip
Thus, it remains to establish inequality~\eqref{3.47}. Its proof is by induction on~$k$. It follows from~\eqref{5.016} and~\eqref{3.28} that 
$$
d^s\bigl(U_0(\theta_{\tau_1}\omega),U_0(\theta_{\tau_2}\omega)\bigr)
\le c\,\bigl|R_{\theta_{\tau_1}\omega}-R_{\theta_{\tau_2}\omega}\bigr|
\le c\,L_\omega|\tau_1-\tau_2|^\alpha. 
$$
Since $V_1^y(\omega)=\psi_1^{y,\sigma_{-1}\omega}(U_0(\sigma_{-1}\omega))$, using~\eqref{3.30} we derive the inequality
$$
d_V^s\bigl(V_1^y(\theta_{\tau_1}\omega),V_1^y(\theta_{\tau_2}\omega)\bigr)
\le K_{\sigma_{-1}\omega}\bigl(1+c\,L_{\sigma_{-1}\omega}\bigr)
|\tau_1-\tau_2|^\alpha,
$$
which coincides with~\eqref{3.47} for $k=1$. Let us assume that~\eqref{3.47} is true for $k=m$ and prove it for $k=m+1$. In view of~\eqref{5.22}, the random finite set~$U_m^y(\omega)$ satisfies the inequality
$$
d^s\bigl(U_m^{y}(\omega_1),U_m^{y}(\omega_2)\bigr)
\le d^s\bigl(V_m^{y}(\sigma_m\omega_1),V_m^{y}(\sigma_m\omega_2)\bigr),
$$ 
where $\omega_i=\theta_{\tau_i}\omega$ for $i=1,2$. It follows that 
$$
d_V^s\bigl(V_{m+1}^y(\omega_1),V_{m+1}^y(\omega_2)\bigr)
\le K_{\sigma_{-1}\omega}\bigl\{|\tau_1-\tau_2|^\alpha+d^s\bigl(V_m^{y}(\sigma_m\omega_1),V_m^{y}(\sigma_m\omega_2)\bigr)\bigr\}. 
$$
The induction hypothesis now implies inequality~\eqref{3.47} with $k=m+1$. The proof of Theorem~\ref{t3.3} is complete. 
\end{proof}

As in the case of Theorem~\ref{t3.1}, inequality~\eqref{3.033} holds for $B=\aA_\omega$ with $C(B)=r$. Furthermore, if the group of shift operators~$\{\sigma_k\}$ is ergodic, then the H\"older exponent in~\eqref{3.31} is a deterministic constant (cf.\ Remark~\ref{r3.3}). Finally, if $\psi_k^{y,\omega}$, $R_\omega^y$, and $K_\omega^y$ do not depend on~$\omega$ for some $y=y_0\in Y$, then the exponential attractor~$\MM_\omega^{y_0}$  constructed in the above theorem is also independent of~$\omega$. Indeed, $\MM_\omega^{y}$ was defined in terms of $\psi_k^{y,\omega}$, $R_\omega^y$, $K_\omega^y$ and the random finite sets $U_k^y(\omega)$ that form $\delta$-nets for the random compact sets~$\CC_k^y(\omega)$. As is mentioned after the proof of Lemma~\ref{l5.4}, these $\delta$-nets are independent of~$\omega$ if so are the random compact sets to be covered. Using this observation, it is easy to prove by recurrence that $\CC_k^{y_0}(\omega)$ and  $U_k^{y_0}(\omega)$ do not depend on~$\omega$, and therefore the same property is true for the attractor~$\MM_\omega^{y_0}$.

\subsection{Exponential attractor for continuous-time RDS}
\label{s3.3}
We now turn to a construction of an exponential attractor for continuous-time RDS. Let us fix a bounded closed interval $Y\subset\R$ and consider a family of RDS $\PPhi^y=\{\varphi_t^{y,\omega}:H\to H, t\ge0\}$, $y\in Y$. We shall always assume that the associated group of shift operators $\theta_t:\Omega\to\Omega$ satisfies the following condition.

\begin{condition} \label{E}
The discrete-time dynamical system~$\{\theta_{k \tau_0}:\Omega\to\Omega, k\in\Z\}$ is ergodic for any $\tau_0>0$. 
\end{condition} 

Given $\tau_0>0$, consider a family of discrete-time RDS $\PPsi^y=\{\psi_k^{y,\omega},k\in\Z_+\}$ defined by
$$
\psi_k^{y,\omega}(u)=\varphi_{k\tau_0}^{y,\omega}(u), \quad u\in H,
\quad k\ge0, \quad \omega\in\Omega.
$$
with the group~$\{\sigma_k=\theta_{k\tau_0}, k\in\Z\}$ as the associated family of shift operators. The following theorem is the main result of this section. 

\begin{theorem} \label{t3.7}
Suppose there is $\tau_0>0$ such that the family~$\{\PPsi^y, y\in Y\}$ satisfies Condition~\ref{UH}, in which $\TT=[-\tau_0,\tau_0]$ and the measurable isomorphism~$\theta_\tau$ coincides with the shift operator. Furthermore, suppose that Condition~\ref{E} is also satisfied, $\aA_\omega^y=B_V(R_\omega^y)$ is a random absorbing set  for~$\PPhi^y$, and the mapping
\begin{equation} \label{HC}
(t,\tau,y,u)\mapsto\varphi_t^{y,\theta_\tau\omega}(u), \quad \R_+\times[-\tau_0,\tau_0]\times Y\times V\to H, 
\end{equation}
is uniformly H\"older continuous on compact subsets with a universal deterministic exponent. Then there is a random compact set $(y,\omega)\mapsto \MM_\omega^y$ in~$H$ with the underlying space $Y\times\Omega$ such that the following properties hold. 

\smallskip
{\bf Attraction.} For any $y\in Y$, the random compact set $\MM_\omega^y$ is an exponential attractor for~$\PPhi^y$. Moreover, the fractal dimension of~$\MM_\omega^y$ is bounded by a universal deterministic constant, and the attraction property holds for the norm of~$V$ uniformly with respect to~$y\in Y$:
\begin{equation} \label{3.51}
d_V\bigl(\varphi_t^{y,\omega}(B),\MM_{\theta_t\omega}^y\bigr)
\le C(B)e^{-\beta t}, \quad t\ge0, \quad y\in Y.
\end{equation}
Here $B\subset H$ is an arbitrary ball, $C(B)$ and~$\beta$ are positive deterministic constants, and the inequality holds with probability~$1$. 

\smallskip
{\bf H\"older continuity.}
The function $(t,y)\mapsto \MM_{\theta_t\omega}^y$ is H\"older-continuous from $Y\times\R$ to the space of random compact sets in~$H$ with the metric~$d_H^s$. More precisely, there is $\gamma\in(0,1]$ such that for any $T>0$ and an almost surely finite random variable $P_{\omega,T}$ we have
\begin{equation} \label{3.52}
d_H^s\bigl(\MM_{\theta_{t_1}\omega}^{y_1},\MM_{\theta_{t_2}\omega}^{y_2}\bigr)
\le P_{\omega,T}\bigl(|t_1-t_2|^\gamma+|y_1-y_2|^\gamma\bigr)
\end{equation}
for $y_1,y_2\in Y$, $t_1,t_2\in[-T,T]$, and $\omega\in\Omega$.
\end{theorem}

\begin{proof}
By rescaling the time, we can assume that $\tau_0=1$. 
Let us denote by~$\{\widetilde\MM_\omega^y\}$ the random compact set constructed in Theorem~\ref{t3.3} for the family of discrete-time RDS~$\PPsi^y$ and define
\begin{equation} \label{rea}
\MM_\omega^y=\bigcup_{\tau\in[0,1]}\varphi_\tau^{y,\theta_{-\tau}\omega}
\bigl(\widetilde\MM_{\theta_{-\tau}\omega}^y\bigr).
\end{equation}
We shall prove that~$\{\MM_\omega^y\}$ possesses all the required properties.

\medskip
{\it Step~1: Measurability}. Let us show that $(y,\omega)\mapsto \MM_\omega^y$ is a random compact set. We need to prove that, for any $u\in H$, the function 
$$
(y,\omega)\mapsto \inf_{v\in \MM_\omega^y}\|u-v\|
$$
is measurable. To this end, we shall apply Proposition~\ref{p5.6} to the family of compact sets 
$$
(y,\omega)\mapsto \KK_{(y,\omega)}=\{(\tau,u)\in[0,1]\times H:u\in \widetilde\MM_{\theta_{-\tau}\omega}^y\}
$$
and the random mapping 
$$
\psi_{(y,\omega)}:[0,1]\times H\to H, \quad (\tau,u)\mapsto\varphi_\tau^{y,\theta_{-\tau}\omega}(u). 
$$
It is straightforward to see that $\MM_\omega^y=\psi_{(y,\omega)}(\KK_{(y,\omega)})$. 
If we prove that~$\psi_{(y,\omega)}$ and~$\KK_{(y,\omega)}$ satisfy the hypotheses of Proposition~\ref{p5.6}, then we can conclude that~$\MM_\omega^y$ is a random compact set in~$H$.

For any fixed $(y,\omega)$, the mapping $(\tau,u)\mapsto \psi_{(y,\omega)}(\tau,u)$ is continuous. On the other hand, the measurability in~$\omega$ and the continuity in~$y$ of the mapping $\varphi_\tau^{y,\theta_{-\tau}\omega}(u)$ imply that, for any fixed $(\tau,u)$, the mapping $\psi_{(y,\omega)}(\tau,u)$ is measurable.  Furthermore, for any $(\tau,u)\in[0,1]\times H$, the mapping 
$$
(y,\omega)\mapsto \inf_{(\tau',u')\in \KK_{(y,\omega)}}\bigl(|\tau-\tau'|+\|u-u'\|\bigr)=\inf_{\tau'\in\Q\cap[0,1]}\Bigl(\,\inf_{u'\in \widetilde\MM_{\theta_{-\tau}\omega}^y}\|u-u'\|\Bigr)
$$
is measurable, so that $\KK_{(y,\omega)}$ is a random compact set. Thus, the application of Proposition~\ref{p5.6} is justified.

\smallskip
{\it Step~2: Semi-invariance}. 
Since~$\{\widetilde\MM_\omega^y\}$ is an exponential attractor for the discrete-time RDS~$\PPsi^y$, for any $y\in Y$ with probability~$1$ we have 
$$
\varphi_k^{y,\omega}(\widetilde\MM_\omega^y)\subset 
\widetilde\MM_{\theta_k\omega}^y, 
\quad k\ge0. 
$$
It follows that, for any rational $s\in\R$ and $y\in Y$, the inequality
\begin{equation} \label{3.54}
\varphi_k^{y,\theta_s\omega}(\widetilde\MM_{\theta_s\omega}^y)\subset 
\widetilde\MM_{\theta_{k+s}\omega}^y, 
\quad k\ge0,
\end{equation}
takes place almost surely. The continuity in~$(s,y)$ of all the objects entering inequality~\eqref{3.54} implies that it holds, with probability~$1$, for all $s\in\R$, $y\in Y$, and $k\ge0$. The semi-invariance can now be established by a standard argument. Namely, for any $\tau\in[0,1]$ and $t\ge0$, we choose an integer $k\ge0$ so that $\sigma=t+\tau-k\in[0,1)$ and write
\begin{align*}
\varphi_t^{y,\omega}\bigl(\varphi_\tau^{y,\theta_{-\tau}\omega}
\bigl(\widetilde\MM_{\theta_{-\tau}\omega}^y\bigr)\bigr)
&=\varphi_{k+\sigma}^{y,\theta_{-\tau}\omega}\bigl(\widetilde\MM_{\theta_{-\tau}\omega}^y\bigr)\\
&=\varphi_\sigma^{y,\theta_{k-\tau}\omega}\bigl(\varphi_k^{y,\theta_{-\tau}\omega}
\bigl(\widetilde\MM_{\theta_{-\tau}\omega}^y\bigr)\bigr)\\
&\subset\varphi_\sigma^{y,\theta_{-\sigma}(\theta_t\omega)}\bigl(\widetilde\MM_{\theta_{k-\tau}\omega}^y\bigr)\\
&=\varphi_\sigma^{y,\theta_{-\sigma}(\theta_t\omega)}\bigl(\widetilde\MM_{\theta_{-\sigma}(\theta_t\omega)}^y\bigr)\subset\MM_{\theta_t\omega}^y,
\end{align*}
where we used~\eqref{3.54} to derive the first inclusion. Since the above relation is true for any~$\tau\in[0,1]$, we conclude that~$\MM_\omega^y$ is semi-invariant under~$\varphi_t^\omega$. 

\smallskip
{\it Step~3: Exponential attraction}. We  first note that, with probability~$1$, 
$$
\sup_{y\in Y}d_V\bigl(\varphi_k^{y,\omega}(\aA_\omega),
\MM_{\theta_k\omega}^y\bigr)\le r\,e^{-\beta k}, \quad k\ge0. 
$$
cf.\ discussion following Theorem~\ref{t3.1}. It follows that 
\begin{equation}  \label{3.53}
\sup_{y\in Y}d_V\bigl(\varphi_k^{y,\theta_s\omega}(\aA_{\theta_s\omega}),
\MM_{\theta_{k+s}\omega}^y\bigr)\le r\,e^{-\beta k}, \quad k\ge0,
\end{equation}
where the inequality holds a.s. for all rational numbers $s\in\R$. The continuity in~$s$ of all the objects entering inequality~\eqref{3.53} implies that, with probability~$1$, it remains true for all $s\in\R$. We now fix an arbitrary ball $B\subset H$ and denote by~$T(B)\ge0$ the instant of time after which the trajectories starting from~$B$ are in~$\aA_{\theta_t\omega}$. For any $t\ge T(B)+1$, we choose $s\in [T(B),T(B)+1)$ such that $k:=t-s$ is an integer and use the cocycle property to write
\begin{align*}
d_V\bigl(\varphi_t^{y,\omega}(B),\MM_{\theta_t\omega}^y\bigr)
&=d_V\bigl(\varphi_k^{y,\theta_s\omega}(\varphi_s^{y,\omega}(B)),
\MM_{\theta_t\omega)}^y\bigr)\\
&\le d_V\bigl(\varphi_k^{y,\theta_s\omega}(\aA_{\theta_s\omega}),
\MM_{\theta_{k+s}\omega)}^y\bigr). 
\end{align*}
Taking the supremum in~$y\in Y$ and using~\eqref{3.53}, we obtain
$$
\sup_{y\in Y}
d_V\bigl(\varphi_t^{y,\omega}(B),\MM_{\theta_t\omega}^y\bigr)
\le r\,e^{-\beta k}\le r\,e^{T(B)+1}e^{-\beta t}.
$$
This proves inequality~\eqref{3.51} with $C(B)=r\,e^{T(B)+1}$.

\smallskip
{\it Step~4: Fractal dimension}. 
As was established in the proof of Theorem~\ref{t3.1}, the fractal dimension of~$\widetilde\MM_\omega^y$ admits the explicit bound (see~\eqref{3.22})
$$
\dim_f(\widetilde\MM_\omega^y)
\le \frac{2^mC\xi_\omega(\ln\xi_\omega+2m)}{m\ln 2},
$$
where $\xi_\omega$ is the random variable defined in~\eqref{3.20}. Since the group~$\{\sigma_k\}$ is ergodic, $\xi_\omega$ is constant, and $\dim_f(\widetilde\MM_\omega^y)$ can be estimated, with probability~$1$, by a  constant not depending on~$y$ and~$\omega$. Since the function $\tau\mapsto\varphi_\tau^{y,\theta_{-\tau}\omega}(u)$ and $\tau\mapsto\widetilde\MM_{\theta_{-\tau}\omega}^y$ are H\"older continuous with a deterministic exponent, it is easy to prove that the fractal dimension of~$\MM_\omega^y$ is bounded by a universal constant. 

\smallskip
{\it Step~5: Time continuity}. Since  mapping~\eqref{HC} is H\"older continuous, the required inequality~\eqref{3.52} will be established if we prove that~\eqref{3.52} is true for~$\widetilde\MM_\omega^y$. However, this is an immediate consequence inequalities~\eqref{3.31}, \eqref{3.031} and the ergodicity of the group of shift operators~$\{\sigma_k\}$. The proof of the theorem is complete. 
\end{proof}

As in the case of discrete-time RDS, if $\varphi_t^{y,\omega}$, $R_\omega^y$, and the random objects entering Condition~\ref{UH} do not depend on~$\omega$ for some~$y_0$, then the exponential attractor~$\MM_\omega^{y_0}$ is also independent of~$\omega$. This fact follows immediately from representation~\eqref{rea}, because $\varphi_{\tau}^{y_0,\theta_{-\tau}\omega}$ and $\widetilde\MM_{\theta_{-\tau}\omega}^{y_0}$ do not depend on~$\omega$. 

\section{Application to a reaction--diffusion system}
\label{s5}
\subsection{Formulation of the main result}
In this section, we apply Theorem~\ref{t3.7} to the reaction-diffusion~\eqref{1.1} in which the amplitude of the random force depends on a parameter. Namely, we consider the equation
\begin{equation} \label{5.1}
\dot u-a\Delta u+f(u)=h(x)+\e\,\eta(t,x),\quad x\in D,
\end{equation}
where $D\subset\R^n$ is a bounded domain with smooth boundary and $\e\in[-1,1]$ is a parameter. Concerning the matrix~$a$, the nonlinear term~$f$, and the external forces~$h$ and~$\eta$, we assume that they satisfy the hypotheses described in Section~\ref{s2.3}, with the stronger condition $p\le\frac{n}{n-2}$ for $n\ge3$. Moreover, we impose a higher regularity on the external force, assuming that
\begin{equation} \label{5.2}
h\in H_0^1(D,\R^k)\cap H^2(D,\R^k), \quad 
\BBBB_3:=\sum_{j=1}^\infty \lambda_j^3b_j^2<\infty,
\end{equation}
where $\lambda_j$ denotes the $j^{\mathrm{th}}$ eigenvalue of the Dirichlet Laplacian. This condition ensures that almost every trajectory of a solution for Eq.~\eqref{5.1} with $f\equiv0$ is a continuous function of time with range in~$H^3$. The following theorem is the main result of this section.

\begin{theorem} \label{t5.1}
Under the above hypotheses, for any $\e\in[-1,1]$ problem~\eqref{5.1}, \eqref{1.2} possesses an exponential attractor~$\MM_\omega^\e$. Moreover, the sets~$\MM_\omega^\e$ can be constructed in such a way that~$\MM_\omega^0$ does not depend on~$\omega$, the fractal dimension of~$\MM_\omega^\e$ is bounded by a universal deterministic constant, the attraction property holds uniformly with respect to~$\e$, and 
\begin{equation} \label{5.3}
d_H^s\bigl(\MM_\omega^{\e_1},\MM_\omega^{\e_2}\bigr)
\le P_\omega|\e_1-\e_2|^\gamma\quad\mbox{for $\e_1,\e_2\in[-1,1]$},
\end{equation}
where $\gamma\in(0,1]$ is a constant and~$P_\omega$ is an almost surely finite random variable. 
\end{theorem}

%In particular, Theorems~A and~B stated in the Introduction are true. 
To prove this result, we shall apply Theorem~\ref{t3.7}. For the reader's convenience, let us describe briefly the conditions we need to check, postponing their verification to the next subsection. 

Recall that $H=L^2$, $V=H_0^1$, and the probability space~$(\Omega,\FF,\IP)$ and the corresponding group of shits operators~$\theta_t$ were defined in Section~\ref{s2.3}. The ergodicity of the restriction of~$\{\theta_t\}$ to any lattice $\tau_0\Z$ is well known (see Condition~\ref{E}), and the Kolmogorov $\e$-entropy of a unit ball in~$V$ regarded as a subset in~$H$ can be estimated by $C\e^{-n}$, where~$n$ is the space dimension (see the fourth item of Condition~\ref{UH}). We shall prove that the following properties are true for a sufficiently large $\tau_0>0$.

\medskip
\noindent
{\bf Absorbing set.} 
There are random variables $R_\omega^\e,R_\omega\ge0$ such that $R_\omega^\e\le R_\omega$ for all $\e\in[-1,1]$, $R\in L^q(\Omega,\IP)$ for any $q\ge1$,  and for any ball $B\subset H$ and a sufficiently large $T(B)>0$ we have
\begin{equation} \label{4.4}
u^{\e,\theta_{\tau}\omega}(t;u_0)\in B_V(R_{\theta_t\omega}^\e)\quad
\mbox{for $t\ge T(B)$, $|\tau|\le\tau_0$, $|\e|\le1$, $u_0\in B$},
\end{equation}
where $u^{\e,\omega}(t;u_0)$ denotes the solution of~\eqref{5.1}, \eqref{1.2}, \eqref{1.3}. Moreover, $R_\omega^\e$ satisfies inequality~\eqref{3.28} with $y_i=\e_i\in[-1,1]$ for an integrable random variable~$L_\omega$ and a deterministic constant $\alpha\in(0,1]$. 

\smallskip
\noindent
{\bf Stability.}
There is $r>0$ such that
\begin{equation} \label{4.5}
u^{\e,\omega}(\tau_0;u_0)\in B_V(R_{\theta_{\tau_0}\omega}^\e)
\quad\mbox{for $|\e|\le1$, $u_0\in \OO_r\bigl(B_V(R_\omega^\e)\bigr)$}. 
\end{equation}

\smallskip
\noindent
{\bf H\"older continuity.}
There is $\alpha>0$ such that, for any $T>0$ and any random variable $r_\omega>0$ all of whose moments are finite,  one can construct a family of random variables $K_\omega^\e\ge1$ satisfying the inequalites
\begin{align} 
&\|u^{\e_1,\theta_{\tau_1}\omega}(t_1;u_{01})
-u^{\e_2,\theta_{\tau_2}\omega}(t_2;u_{02})\|\notag\\
& \qquad\qquad\le K_\omega^{\e_1,\e_2}\bigl(|\e_1-\e_2|
+|\tau_1-\tau_2|^\alpha+\|u_{01}-u_{02}\|+|t_1-t_2|^\alpha\bigr),\label{4.6}\\
&\|u^{\e_1,\theta_{\tau_1}\omega}(\tau_0;u_{01})
-u^{\e_2,\theta_{\tau_2}\omega}(\tau_0;u_{02})\|_1\notag\\
& \qquad\qquad \le K_\omega^{\e_1,\e_2}\bigl(|\e_1-\e_2|
+|\tau_1-\tau_2|^\alpha+\|\tilde u_{01}-\tilde u_{02}\|\bigr),\label{4.60}
\end{align}
where $|\e_i|\le1$, $|\tau_i|\le\tau_0$, $0\le t_i\le T$, $u_{0i}\in B_V(r_\omega)$, $\tilde u_{0i}\in B_H(r_\omega)$, and we set $K_\omega^{\e_1,\e_2}=\max\{K_\omega^{\e_1},K_\omega^{\e_1}\}$. Moreover, there is a random variable $K_\omega$ belonging to $L^q(\Omega,\IP)$  for any $q\ge1$ such that $K_\omega^\e\le K_\omega$ for all $\e\in[-1,1]$. 

\medskip
We shall also prove that the random variables~$R_\omega^0$ and~$K_\omega^0$ are constants. If these properties are established, then all the hypotheses of Theorem~\ref{t3.7} are fulfilled, and its application to the RDS associated with problem~\eqref{5.1}, \eqref{1.2} gives the conclusions of Theorem~\ref{t5.1}. 

\subsection{Proof of Theorem~\ref{t5.1}}
\label{s5.2}
{\it Step~1: Absorbing set}.
Let $U^{\e,\omega}(t)$ be the unique stationary solution of the equation
\begin{equation} \label{4.7}
\dot u-a\Delta u=\e\,\eta(t,x), \quad t\in\R,
\end{equation}
supplemented with the Dirichlet boundary condition~\eqref{1.2}. It is straightforward to see that, with probability~$1$,
\begin{equation} \label{4.8}
U^{\e,\theta_\tau\omega}(t)=\e\,U^{\omega}(t+\tau), \quad t,\tau\in\R,
\end{equation}
where $U^{\omega}(t)=U^{1,\omega}(t)$. Using the It\^o formula and the regularity assumption~\eqref{5.2}, one can prove that\,\footnote{For instance, see Proposition~2.4.10 in~\cite{KS2012} for the more complicated case of the Navier--Stokes system.}
\begin{equation} \label{4.9}
\E\,e^{\delta\sup_t M_t}<\infty, \quad 
M_t(\omega):=\|U^\omega(t)\|_3^2+\biggl|\int_0^t\|U^\omega(s)\|_4^2\,ds\biggr|-C(1+|t|),
\end{equation}
where $\delta>0$ and~$C>0$ are deterministic constant, and the supremum is taken over $t\in\R$. Moreover, by Proposition~\ref{c5.7}, inequality~\eqref{6.30} holds for~$U$. 

Solutions of~\eqref{5.1}, \eqref{1.2} can be written as 
\begin{equation} \label{4.10}
u^{\e,\theta_\tau\omega}(t,x)=\e\,U^{\omega}(t+\tau,x)+v^{\e,\tau,\omega}(t,x),
\end{equation}
where $v=v^{\e,\tau,\omega}$ is the solution of the problem
\begin{align}
\dot v-a\Delta v+f(v+\e\,U^\omega(t+\tau))&=h(x), \label{4.11}\\
v\bigr|_{\p D}&=0,\label{4.12}\\
v(0,x)&=v_0(x), \label{4.13}
\end{align}
where $v_0(x)=u_0(x)-\e\,U^\omega(\tau,x)$. In what follows, we shall often omit the subscripts~$\e$ and~$\omega$ to simplify notation. We wish to derive some a priori estimates for~$v$. Since the corresponding argument is rather standard, we only sketch it. 

Taking the scalar product of~\eqref{4.11} in~$L^2$ and carrying out some transformations, we derive
\begin{equation} \label{4.14}
\p_t\|v\|^2+c_1\bigl(\|v\|^2+\|v\|_1^2+\|v\|_{L^{p+1}}^{p+1}\bigr)
\le C_1\bigl(1+\|h\|_{-1}^2+\|\e\,U\|_{L^{p+1}}^{p+1}\bigr),
\end{equation}
where $U=U^\omega(\cdot+\tau)$, and we used inequalities~\eqref{2.5}, \eqref{2.6}, and~\eqref{2.8}. Let us fix any $\delta\in(0,c_1)$. Applying the Gronwall inequality, using the continuity of the embedding $H^1\subset L^{p+1}$, and recalling that $|\tau|\le \tau_0$, we obtain
\begin{equation} \label{4.15}
\|v(t)\|^2+c_1\int_0^te^{-c_1(t-\sigma)}\bigl(\|v\|_1^2+\|v\|_{L^{p+1}}^{p+1}\bigr)\,d\sigma
\le e^{-c_1 t}\|v_0\|^2+R_{\theta_t\omega}^{\e,1},
\end{equation}
where we set 
$$
R_\omega^{\e,1}=
C_1\int_{-\infty}^0e^{\delta\sigma}\bigl(1+\|h\|_{-1}^2
+\|\e\,U^\omega(\sigma+\tau_0)\|_{1}^{p+1}\bigr)\,d\sigma. 
$$
%Note that all the moments of~$R^1$ are finite in view of~\eqref{4.9}. 

We now derive a similar estimate for the $H^1$ norm of~$v$. Taking the scalar product of~\eqref{4.11}  with $-2(t-s)\Delta v$ in~$L^2$, after some transformations we derive
\begin{multline*}
\p_t\bigl((t-s)\|\nabla v\|^2\bigr)+c_2(t-s)\,\|v\|_2^2\\
\le \|\nabla v\|^2
+C_2(t-s)\bigl(1+\|h\|^2+\|v\|_{L^{p+1}}^{p+1}+\|\e\,U\|_2^{p+1}\bigr). 
\end{multline*}
Integrating in $t\in(s,s+1)$, we obtain 
\begin{multline*} 
\|\nabla v(s+1)\|^2+c_2\int_s^{s+1}(\sigma-s)\|\Delta v\|^2d\sigma\\
\le C_3+\int_s^{s+1}\bigl(\|\nabla v\|^2+C_2\|v\|_{L^{p+1}}^{p+1}\bigr)d\sigma
+C_2\int_s^{s+1}\|\e\,U\|_2^{p+1}d\sigma,
\end{multline*}
where $C_3=C_2(1+\|h\|^2)$. Taking $s=t-1$ and using~\eqref{4.15} to estimate the second term on the right-hand side, we obtain 
\begin{equation} \label{4.16}
\|v(t)\|_1^2\le C_3+C_4\bigl(e^{-c_1 t}\|v_0\|^2+R_{\theta_t\omega}^{\e,1}\bigr)+C_2 R_{\theta_t\omega}^{\e,2},\quad t\ge1,
\end{equation}
where we  set\,\footnote{To have an absorbing set, one could take for $R_\omega^{\e,2}$ the integral of $\|\e\,U^\omega(\sigma+\tau_0)\|_2^{p+1}$ in $\sigma\in[-3,0]$. However, in this case the stability condition may not hold, and therefore we define~$R_\omega^{\e,2}$ in a different way. Our choice ensures that~\eqref{4.20} holds for the radius of the absorbing ball.}
$$
R_\omega^{\e,2}=\int_{-\infty}^{0}e^{\delta(\sigma+3)}\|\e\,U^\omega(\sigma+\tau_0)\|_2^{p+1}d\sigma.
$$
Let us define $R_\omega^\e$ by the relation 
\begin{equation} \label{4.17}
(R_\omega^\e)^2=8\Bigl(1+C_3+C_4R_{\omega}^{\e,1}
+C_2R_{\omega}^{\e,2}+\sup_{\sigma\le0}\bigl(e^{\delta(\sigma+2\tau_0)}\|\e\,U^\omega(\sigma+\tau_0)\|_1^2\bigr)\Bigr)
\end{equation}
and set $R_\omega=R_\omega^1$. It is clear that $R_\omega^\e\le R_\omega$ for all $\e\in[-1,1]$. 
Relations~\eqref{4.10} and~\eqref{4.16} imply that
\begin{equation} \label{4.18}
\|u^{\e,\theta_\tau\omega}(t)\|_1^2
\le 2C_4e^{-c_1 t}\bigl(\|u_0\|^2+\|\e\,U^\omega(\tau)\|^2\bigr)
+4^{-1}(R_{\theta_t\omega}^\e)^2-1, \quad t\ge1,
\end{equation}
whence we see~\eqref{4.4} holds for any ball $B\subset H$ and a sufficiently large $T(B)>0$. Moreover, it follows from~\eqref{4.9} and~\eqref{4.17} that all the moments of~$R_\omega$ are finite. Finally, Proposition~\ref{c5.7} and the stationarity of~$U$ imply that~$R_\omega^\e$ satisfies~\eqref{3.28} with a constant~$\alpha\in(0,1/2)$ and an integrable random variable~$L_\omega$. 

\smallskip
{\it Step~2: Stability}.
It follows from~\eqref{4.17} that the stability property~\eqref{4.5} with parameters $r>0$ and~$\tau_0>0$ will certainly be satisfied if 
\begin{equation} \label{4.19}
2C_4e^{-c_1 \tau_0}\bigl((R_\omega^\e+r)^2+\|\e\,U^\omega(\tau)\|^2\bigr)
+4^{-1}(R_{\theta_{\tau_0}\omega}^\e)^2-1\le (R_{\theta_{\tau_0}\omega}^\e)^2. 
\end{equation}
Let us note that 
\begin{equation} \label{4.20}
(R_{\theta_t\omega}^\e)^2\ge e^{-\delta t}(R_\omega^\e)^2, \quad t\ge0. 
\end{equation}
We now take an arbitrary $r>0$ and choose $\tau_0>0$ so large that
$$
4C_4e^{-c_1\tau_0}r^2\le 1, \quad 16 C_4e^{-(c_1-\delta)\tau_0}\le1.
$$
In this case, inequality~\eqref{4.19} holds, so that the stability condition is fulfilled.

\smallskip
{\it Step~3: H\"older continuity}.
Representation~\eqref{4.10} implies that it suffices to establish analogues  of~\eqref{4.6} and~\eqref{4.60} for solutions of problem~\eqref{4.11}--\eqref{4.13}. Namely, we first prove that for any random variable $r_\omega>0$ with finite moments there is a family of almost surely finite random variables~${\widetilde K}_\omega^\e$ such that
\begin{align}
\|v_1(t)-v_2(t)\|
& \le {\widetilde K}_\omega^{\e_1,\e_2}\bigl(\|\e_1U^{\omega_1}-\e_2U^{\omega_2}\|_{L^\infty(0,T;H^1)}+\|v_{01}-v_{02}\|\bigr),\label{4.21}\\
\|v_1(\tau_0)-v_2(\tau_0)\|_1
& \le {\widetilde K}_\omega^{\e_1,\e_2}\bigl(\|\e_1U^{\omega_1}-\e_2U^{\omega_2}\|_{L^\infty(0,\tau_0;H^2)}+\|v_{01}-v_{02}\|\bigr),
\label{4.22}
\end{align}
where $|\e_i|\le1$, $|\tau_i|\le\tau_0$, $0\le t\le T$, $v_{0i}\in B_H(r_\omega)$, and we set $\omega_i=\theta_{\tau_i}\omega$, $v_i(t)=v^{\e_i,\tau_i,\omega}$, and ${\widetilde K}_\omega^{\e_1,\e_2}=\max\{{\widetilde K}_\omega^{\e_1}, {\widetilde K}_\omega^{\e_2}\}$. Moreover, our proof will imply that ${\widetilde K}_\omega^\e\le {\widetilde K}_\omega$ for all $\e\in[-1,1]$, where the random constant ${\widetilde K}$ belongs to $L^q(\Omega,\IP)$ for any $q\ge1$. 
Once these properties are established, the H\"older continuity of~$U^\omega(t)$  and relations~\eqref{4.9} and~\eqref{4.10} will prove inequalities~\eqref{4.6} and~\eqref{4.60} with $t_1=t_2$. We shall next show that the solutions of~\eqref{4.11}--\eqref{4.13} with $v_0\in B_V(r_\omega)$ satisfy the inequality
\begin{equation} \label{4.23}
\|v^{\e,\tau,\omega}(t_1;v_0)-v^{\e,\tau,\omega}(t_2;v_0)\|
\le {\widetilde K}_\omega^\e|t_1-t_2|^{\alpha}, \quad t_1,t_2\in[0,T],
\end{equation}
with possibly a larger random constant ${\widetilde K}_\omega^\e$ with the same property. This will complete the proof of the property of H\"older continuity and that of Theorem~\ref{t5.1}. 

\medskip
We begin with~\eqref{4.21}. To simplify the presentation, we shall assume that $n\ge3$. In what follows, we denote by $\{K_\omega^{\e,i}, \e\in[-1,1]\}$ (where $i\ge1$) families of random variables that can be bounded by a random constant belonging to $L^q(\Omega,\IP)$ for any $q\ge1$.  The difference $v=v_1-v_2$ satisfies the equation
\begin{equation} \label{4.25}
\dot v-a\Delta v+f(u_1)-f(u_2)=0
\end{equation}
and the boundary and initial conditions~\eqref{4.12} and~\eqref{4.13}, where  $u_i=v_i+\e_i U^{\omega^i}$ and $v_0=v_{01}-v_{02}$.  Taking the scalar product of~\eqref{4.25} with $2v$ in $L^2$ and using the ``monotonicity'' assumption~\eqref{2.7}, we derive
\begin{align*}
\p_t\|v\|^2+2c_3\|\nabla v\|^2&\le -\bigl(f(u_1)-f(u_2),v\bigr)\\
&\le C\|v\|^2+C_5\bigl(\|\xi\|_{L^q}+\||u_1|^{p-1}\xi\|_{L^q}+\||u_2|^{p-1}\xi\|_{L^q}\bigr)
\|v\|_{1}\\
&\le C\|v\|^2+c_3\|\nabla v\|^2+C_6\bigl(1+\|u_1\|_{L^{p+1}}^{2(p-1)}
+\|u_2\|_{L^{p+1}}^{2(p-1)}\bigr)\|\xi\|_1^2,
\end{align*}
where $q=\frac{2n}{n+2}$ and $\xi=\e_1U^{\omega_1}-\e_2U^{\omega_2}$. Applying the Gronwall inequality, we obtain
\begin{equation} \label{4.26}
\|v(t)\|^2+c_3\int_0^te^{C(t-\sigma)}\|\nabla v\|^2d\sigma
\le e^{Ct}\|v_0\|^2+C_7\max\{K_\omega^{\e,1},K_\omega^{\e,2}\}\|\xi\|_{L^\infty(0,T;H^1)}^2,
\end{equation}
where $0\le t\le T$, $C_7=C_7(T)$, and we set
$$
K_\omega^{\e,1}=\int_0^T\bigl(1+\|u^{\e,\omega}(\sigma)\|_{L^{p+1}}^{p+1}\bigr)\,d\sigma.
$$
Inequality~\eqref{4.26} immediately  implies~\eqref{4.21}. 

\smallskip
To prove~\eqref{4.22}, we first note that, in view of~\eqref{4.26}, there is a measurable function $s:\Omega\to\R$ such that, with probability~$1$, we have $s_\omega\in[\frac{\tau_0}{4}, \frac{3\tau_0}{4}]$  and 
\begin{equation} \label{4.207}
\|\nabla v(s_\omega)\|^2\le C_8\bigl(\|v_0\|^2+K_\omega^{\e,1}\|\xi\|_{L^\infty(0,\tau_0;H^1)}^2\bigr). 
\end{equation}
Let us take the scalar product of~\eqref{4.25} with $-2\Delta v$ in~$L^2$. After some transformations, we obtain 
\begin{equation} \label{4.27}
\p_t\|\nabla v\|^2+2c_4\,\|\Delta v\|^2
\le C_8
\bigl(1+\|u_1\|_{L^{n(p-1)}}^{p-1}+\|u_2\|_{L^{n(p-1)}}^{p-1}\bigr)
\|v+\xi\|_{L^q}\|\Delta v\|,
\end{equation}
where $q=\frac{2n}{n-2}$. Since $H^1\subset L^q$ and $H^1\subset L^{n(p-1)}$, applying the interpolation and Cauchy--Schwartz inequalities, from~\eqref{4.27} we derive
$$
\p_t\|\nabla v(t)\|^2+c_4\,\|\Delta v\|^2\le 
C_9\bigl(1+\|u_1\|_1^{4(p-1)}+\|u_2\|_1^{4(p-1)}\bigr)\,(\|v\|^2+\|\xi\|_1^2).
$$
Integrating in $t\in[s_\omega,\tau_0]$ and using~\eqref{4.26}, \eqref{4.207} and~\eqref{4.18} (we can assume that $\tau_0\ge4$), we obtain~\eqref{4.22}. 

\smallskip
It remains to establish inequality~\eqref{4.23}. We shall only outline its proof. Taking the scalar product of~\eqref{4.11} with $-2\Delta v$ and using some standard arguments (cf.\ derivation of~\eqref{4.16}), we obtain
$$
\int_0^T\|\Delta v\|^2d\sigma\le C_{10}\|v_0\|_1^2+K_\omega^{\e,2}.
$$
Combining this with~\eqref{4.9} and~\eqref{4.11}, we see that 
$$
\int_0^T\|\dot v\|^2d\sigma\le K_\omega^{\e,3}. 
$$
It follows that $v$ is H\"older continuous with the exponent~$1/2$. Since~$U^\omega$ is also H\"older continuous in time, in view of~\eqref{4.10} we arrive at the required result. Finally, it is not diffucult to see that the random variables~$R_\omega^0$ and~$K_\omega^0$ are constant. The proof of the theorem is complete.

\section{Appendix}
\label{s6}
\subsection{Coverings for random compact sets}
\label{s6.1}
In this section, we have gathered three auxiliary results on coverings of  random compact sets by balls centred at the points of  random finite sets. The first of them establishes the existence of a ``minimal'' covering with an explicit bound of the number of balls in terms of the Kolmogorov $\e$-entropy of the random compact set in question.

\begin{lemma} \label{l6.1}
Let $\{\aA_\omega\}$ be a random compact set in a Hilbert space~$H$. Then  for any measurable function~$\delta=\delta_\omega$ satisfying the inequality $0<\delta\le1$ one can construct a random finite set $U_\delta(\omega)\subset H$ such that for 
\begin{align} 
d^s\bigl(\aA_\omega,U_\delta(\omega)\bigr)
&\le\delta_\omega,\label{6.1}\\
\ln\bigl(\#U_{\delta}(\omega)\bigr)
&\le \HH_{\delta_\omega/2}(\aA_\omega,H).
\label{6.2}
\end{align}
Moreover, if~$\delta_\omega\equiv\delta$ is constant, then one can replace~$\delta_\omega/2$ in the right-hand side~\eqref{6.2} by~$\delta$. 
\end{lemma}

Note that inequality~\eqref{6.1} is equivalent to the inclusions
\begin{equation} \label{6.01}
\aA_\omega\subset \bigcup_{u\in U_\delta(\omega)} B_H(u,\delta_\omega), \quad
U_{\delta}(\omega)\subset\OO_{\delta_\omega}(\aA_\omega).
\end{equation}

\begin{proof}
We first assume that~$\delta_\omega\equiv\delta$. Let $\{u_k\}\subset H$ be a dense sequence. For any $\kkk=\{k_1,\dots,k_n\}\subset\N$, define the random varable
$$
Z_\omega(\kkk)=
\left\{
\begin{array}{cl}
1,&\aA_\omega\subset\bigcup\limits_{i=1}^n B(u_{k_i},\delta),\\[4pt]
0,&\mbox{otherwise}. 
\end{array}
\right.
$$
Since~$\aA_\omega$ is a (random) compact set, for any $\omega$ there is a finite subset $\kkk\subset\N$ such that $Z_\omega(\kkk)=1$.  Let~$\Omega_n$ be the set of those~$\omega\in\Omega$ for which there is an $n$-tuple $\kkk\subset\N$ such that $Z_\omega(\kkk)=1$ and $Z_\omega(\kkk')=0$ for any subset $\kkk'\subset\N$ containing less than~$n$ elements. Then we have $\Omega=\cup_{n\ge1}\Omega_n$. Furthermore, since $\Omega_n$ is the intersection of the measurable sets
%for any open subset $U\subset H$ the set $\{\omega\in\Omega:\aA_\omega\cap U\ne\varnothing\}$ is measurable, 
$$
\bigcap_{\#\kkk=n-1}\{Z_\omega(\kkk)=0\}\quad\mbox{and}\quad
\bigcup_{\#\kkk=n}\{Z_\omega(\kkk)=1\}, 
$$
we have $\Omega_n\in\FF$ for any $n\ge1$. Thus, it suffices to construct~$U_\delta$ on each subset~$\Omega_n$. 

Indexing the set of all $n$-tuples $\kkk\subset\N$ in an arbitrary way, it is easy to construct measurable functions $I_k:\Omega_n\to\{0,1\}$ such that, for any $\omega\in\Omega_n$, we have  
\begin{equation} \label{6.3}
\#\kkk(\omega)=n, \quad 
\aA_\omega\subset\bigcup\limits_{k\in\kkk(\omega)} B_H(u_k,\delta),
\quad B_H(u_k,\delta)\cap \aA_\omega\ne\varnothing,
\end{equation}
where $\kkk(\omega)=\{k\in\N:I_k(\omega)=1\}$ and $k\in\kkk(\omega)$ in the third relation. We claim that $U_\delta(\omega)=\{u_k,k\in\kkk(\omega)\}$ satisfies the required properties. Indeed, for any $u\in H$, we have 
$$
d(u,U_\delta(\omega))=\min\{\|u-u_k\|:I_k(\omega)=1\}, 
\quad \omega\in\Omega_n,
$$
whence it follows easily that $U_\delta(\omega)$ is a random finite set. Furthermore, inclusions~\eqref{6.01} (which are equivalent to inequality~\eqref{6.1}) are consequences of the second and third relations in~\eqref{6.3}. Let us prove that inequality~\eqref{6.2} holds with~$\delta_\omega/2$ replaced by~$\delta_\omega$; that is,
\begin{equation} \label{6.4}
\ln n\le \HH_{\delta}(\aA_\omega,H)
\quad\mbox{for $\omega\in\Omega_n$}. 
\end{equation}
To see this, note that the set~$\aA_\omega$ admits a covering by balls~$\{B_j\}$  such that
$$
\ln\bigl(\#\{B_j\}\bigr)\le \HH_{\delta}(\aA_\omega,H), \quad \diam(B_j)\le\delta. 
$$
Choosing arbitrary points $u_{k_j}$ in every ball~$B_j$, we see that one can cover~$\aA_\omega$ by the balls $\{B_H(u_{k_j},\delta)\}$. The choice of~$n$ now implies that $n\le\#\{B_j\}$, whence it follows that~\eqref{6.4} holds.

We now turn to the case of an arbitrary function~$\delta_\omega$ such that $0<\delta_\omega\le1$. Let us define $\Omega^{(k)}=\{\omega\in\Omega:2^{-k}<\delta_\omega\le2^{1-k}\}$, so that $\Omega=\cup_{k\ge1}\Omega^{(k)}$. In view of what has been proved above, on each~$\Omega^{(k)}$ one can construct a random finite set $U_k(\omega)$ such that, for $\omega\in\Omega^{(k)}$, we have
\begin{equation*} 
d^s\bigl(\aA_\omega,U_k(\omega)\bigr)\le 2^{-k}, \quad
\ln\bigl(\#U_k(\omega)\bigr)\le \HH_{2^{-k}}(\aA_\omega,H). 
\end{equation*}
Setting $U_\delta(\omega)=U_k(\omega)$ for $\omega\in\Omega^{(k)}$, we obtain the required covering. The proof of the lemma is complete.
\end{proof}

The second result shows that, if a random compact set depends on a parameter in a Lipschitz manner, then the random finite set constructed above can be chosen to have a similar dependence on the parameter. To prove it, we shall need the following auxiliary construction. 

Let us denote by~$\Delta_n\subset\R^n$ the set of vectors $\theta=(\theta_1,\dots,\theta_n)$ such that $\theta_i\ge0$ and $\sum_i\theta_i=1$. 
Given subsets $W_i\subset H$, $1\le i\le n$, a vector $\theta\in\Delta_n$, and a number~$\alpha>0$, we define
$$
[W_1,\dots,W_n]_\theta^\alpha=\biggl\{\sum_{i=1}^n\theta_iu_i:u_i\in W_i, \|u_i-u_j\|_H\le \alpha\mbox{ for $1\le i,j\le n$}\biggr\}. 
$$
It is straightforward to check that
\begin{align} 
\ln \bigl(\#[W_1,\dots,W_n]_\theta^\alpha\bigr)
&\le \ln(\#W_1)+\cdots+\ln(\#W_n),\label{6.15}\\
d^s\bigl([W_1,\dots,W_n]_{\theta^1}^\alpha,[W_1,\dots,W_n]_{\theta^2}^\alpha\bigr)
&\le\alpha |\theta^1-\theta^2|,
\label{6.16} 
\end{align}
where $\theta^j=(\theta_1^j,\dots,\theta_n^j)$ and $|\theta^1-\theta^2|=\max_i|\theta_i^1-\theta_i^2|$. 
Moreover, if~$\aA\subset H$ and $r_i\ge0$ are such that
$$
\aA\subset \bigcup_{u\in W_i}B_H(u,r_i), \quad 1\le i\le n,
$$
then for any $\theta\in\Delta_n$ we have
\begin{equation} \label{6.17}
\aA\subset\bigcup_{u\in [W_1,\dots,W_n]_{\theta}^{r}}
B_H(u,\max\{r_i,1\le i\le n\}),
\end{equation}
where $r=\max\{r_i+r_j,1\le i,j\le n\}$. 

\begin{proposition} \label{p6.2}
Let $Y\subset\R$ be a closed interval and let $\{\aA_\omega^y,y\in Y\}$ be a family of random compact sets in a Hilbert space~$H$  such that 
\begin{equation} 
%\HH_\e(\aA_\omega^y,H)&\le C_1\e^{-m}
%\quad\mbox{for $y\in Y$, $\e\in(0,1]$},\label{6.7}\\
d^s(\aA_\omega^{y_1},\aA_\omega^{y_2})
\le C\,|y_1-y_2|\quad\mbox{for $y_1,y_2\in Y$},\label{6.6}
\end{equation}
where~$C\ge1$  is a  finite random constant. Then there exists a random finite set $(\delta,y,\omega)\mapsto U_{\delta,y}(\omega)$ with the underlying space $(0,1]\times Y\times\Omega$ such that 
\begin{align}
d^s\bigl(\aA_\omega^y,U_{\delta,y}(\omega)\bigr)&\le \delta,\label{6.8}\\
\ln\bigl(\#U_{\delta,y}(\omega)\bigr)
&\le 4\,\HH_{2^{-4}\delta}(\aA_\omega^y,H),\label{6.9}\\
d^s\bigl(U_{\delta_1,y_1}(\omega),U_{\delta_2,y_2}(\omega)\bigr)
&\le c\,\bigl(|\delta_1-\delta_2|+C\, |y_1-y_2|\bigr),
\label{6.10}
\end{align}
where $y,y_1,y_2\in Y$, $\delta,\delta_1,\delta_2\in(0,1]$, and $c\ge1$ is an absolute constant.
\end{proposition}

In particular, taking a measurable function $\delta=\delta_\omega$ with range in~$(0,1]$, we can construct a random finite set~$(y,\omega)\mapsto U_{\delta,y}(\omega)$ such that 
\begin{equation}
d^s\bigl(U_{\delta,y_1}(\omega),U_{\delta,y_2}(\omega)\bigr)
\le c\,C\, |y_1-y_2|,\label{5.31}
\end{equation}
and inequalities~\eqref{6.8} and~\eqref{6.9} hold with $\delta=\delta_\omega$ in the right-hand side.

The proof given below will imply that if~$\aA_\omega^y$ does not depend on~$\omega$ for some $y=y_0$, then the random set~$U_{\delta,y}(\omega)$ satisfying \eqref{6.8}--\eqref{6.10} can be chosen in such a way that~$U_{\delta,y_0}(\omega)$ is also independent of~$\omega$. Furthermore, if~$\aA_\omega^y$ does not depend on~$\omega$ for all $y\in Y$, then~$U_{\delta,y}$ is also independent of~$\omega$. The latter observation implies the following corollary used in the main text. 

\begin{corollary} \label{c5.3}
Let $V\subset H$ be two Hilbert spaces with compact embedding. Then there is a random finite set $(\delta,R)\mapsto U_{\delta,R}$ with the underlying space $(0,1]\times \R_+$ such that 
\begin{align}
d_H^s\bigl(B_V(R),U_{\delta,R}\bigr)&\le\delta,\label{5.014}\\
\ln\bigl(\#U_{\delta,R}\bigr)&\le 4\,\HH_{{\delta}/{16R}}(V,H),\label{5.015}\\
d_H^s(U_{\delta,R_1},U_{\delta,R_2})&\le c\,|R_1-R_2|,\label{5.016}
\end{align}
where  $R,R_1,R_2\ge0$ and $\delta\in(0,1]$ are arbitrary, and $c>0$ is an absolute constant. 
\end{corollary}

To prove this result, it suffices to apply Proposition~\ref{p6.2} to the non-random compact set~$B_V(R)$ depending on the parameter $R\in\R_+$. 

\begin{proof}[Proof of Proposition~\ref{p6.2}]
%To simplify the presentation, we confine ourselves to the case $d=1$; the proof in the general case is based on a similar argument. 
Without loss of generality, we assume that the random variable~$C$ is  constant, since one can represent~$\Omega$ as the union of the subsets $\Omega_{l}=\{\omega\in\Omega: l\le C<l+1\}$ and construct required random finite sets on each~$\Omega_{l}$. 

Let us fix an integer $k\ge1$ and denote by $\nu_k<C^{-1}2^{-k-4}$ the largest number such that~$N_k:=\nu_k^{-1}$ is an integer. We now set $y_j^k=j\nu_k$ for $j\in\Z_+$. In view of Lemma~\ref{l6.1}, there are random finite sets $U_j^k(\omega)\subset H$ such that
\begin{align}
d^s\bigl(\aA_\omega^{y_j},U_j^k(\omega)\bigr)&\le 2^{-k-3},\label{5.14}\\
\ln\bigl(\#U_j^k(\omega)\bigr)
&\le \HH_{2^{-k-3}}(\aA_\omega^{y_j},H),\label{5.15}
\end{align}
where we write $y_j$ instead of~$y_j^k$ to simplify the notation. We now need the following lemma, whose proof\,\footnote{We thank A.~Iftimovici for the simple geometric argument proving Lemma~\ref{l5.3}.} is given at end of this section.

\begin{lemma} \label{l5.3}
Let $A_1,\dots,A_4$ be the vertices of a rectangle $\Pi\subset\R^2$. Then there are Lipschitz functions $\theta_i:\Pi\to[0,1]$, $1\le i\le 4$, such that
$$
\sum_{i=1}^4\theta_i(A)=1, \quad \sum_{i=1}^4\theta_i(A)A_i=A
\quad\mbox{for $A\in \Pi$}. 
$$
\end{lemma}

Let $\theta_i(A)$, $1\le i\le 4$, be the functions constructed in Lemma~\ref{l5.3} for the rectangle $\Pi=[2^{-k},2^{1-k}]\times[y_j,y_{j+1}]$. For $2^{-k}<\delta\le 2^{1-k}$ and $y_{j}\le y\le y_{j+1}$, denote by $A_{\delta,y}\in\Pi$ the point with  the coordinates~$(\delta,y)$. Let us define
$$
U_{\delta,y}(\omega)
=[U_j^k,U_j^{k+1},U_{j+1}^k,U_{j+1}^{k+1}]_{\theta(\delta,y)}^{2^{-k-1}},
$$
where $\theta(\delta,y)=(\theta_i(A_{\delta,y}), 1\le i\le 4)\in\Delta_4$. We claim that~$U_{\delta,y}(\omega)$ satisfies the required properties. 

\smallskip
%inequality~\eqref{5.14} implies that 
%\begin{equation} \label{5.16}
%\aA_\omega^{y_j}\subset\bigcup_{u\in U_j^k(\omega)}B_H(u,2^{-k-2}).
%\end{equation}
%Furthermore, 
Indeed, it follows from the choice of~$y_j$ that
\begin{equation} \label{5.19}
d^s(\aA_\omega^{y_j},\aA_\omega^{y})\le 2^{-k-4}
\quad\mbox{for $y_j\le y\le y_{j+1}$}.
\end{equation}
Combining this with~\eqref{5.14}, we see that 
\begin{equation} \label{5.17}
d^s\bigl(\aA_\omega^y,U_{j}^k(\omega)\bigr)\le 2^{-k-2}.
\end{equation}
Inclusion~\eqref{6.17} now implies that 
\begin{equation} \label{5.18}
d\bigl(\aA_\omega^{y},U_{\delta,y}(\omega)\bigr)\le 2^{-k-2}\le\delta/4.
\end{equation}
On the other hand, the definition of~$U_{\delta,y}(\omega)$ and inequality~\eqref{5.17} imply that 
$$
d\bigl(U_{\delta,y}(\omega),\aA_\omega^{y}\bigr)\le 2^{-k}\le\delta.
$$
Combining this with~\eqref{5.18}, we obtain~\eqref{6.8}. 

Inequality~\eqref{5.19} implies that an $\e$-covering for~$\aA_\omega^{y}$ with $y_j\le y\le y_{j+1}$ is an $(\e+2^{-k-4})$-covering for~$\aA_\omega^{y_j}$. Taking $\e=2^{-k-4}$, we see that
$$
\HH_{2^{-k-3}}(\aA_\omega^{y_j},H)\le \HH_{2^{-k-4}}(\aA_\omega^{y},H). 
$$
Combining this with~\eqref{5.15} and~\eqref{6.15}, we obtain~\eqref{6.9}:
$$
\ln\bigl(\#U_{\delta,y}(\omega)\bigr)
\le 4\HH_{2^{-k-3}}(\aA_\omega^{y_j},H)
\le 4\HH_{2^{-k-4}}(\aA_\omega^{y},H)
\le 4\HH_{2^{-4}\delta}(\aA_\omega^{y},H).
$$
Finally, inequality~\eqref{6.10} follows from~\eqref{6.16} and the explicit form of the functions~$\theta_i(A)$ (see~\eqref{5.20}):
\begin{align*}
d^s\bigl(U_{\delta_1,y_1},U_{\delta_2,y_2}\bigr)
&\le 2^{-k-1}|\theta(A_{\delta_1,y_1})-\theta(A_{\delta_2,y_2})|\\
&\le 2^{-k-1}(\nu_k 2^{-k})^{-1}
\bigl(\nu_k|\delta_1-\delta_2|+2^{-k}|y_1-y_2|\bigr)\\
&\le \tfrac12\,|\delta_1-\delta_2|+8C\,|y_1-y_2|. 
\end{align*}
The proof of the proposition is complete. 
\end{proof}

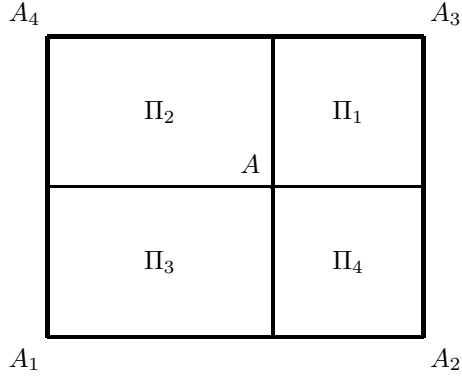
\begin{figure}[t]
\hskip2,5cm
\ifx\JPicScale\undefined\def\JPicScale{1}\fi
\unitlength \JPicScale mm
\begin{picture}(65,40)(0,0)

\linethickness{0.5mm}
\put(10,3){\line(0,1){40}}

\linethickness{0.5mm}
\put(10,3){\line(1,0){50}}

\linethickness{0.5mm}
\put(60,3){\line(0,1){40}}

\linethickness{0.5mm}
\put(10,43){\line(1,0){50}}
\put(40,23){\circle*{0.8}}

\linethickness{0.3mm}
\put(40,3){\line(0,1){40}}
\linethickness{0.3mm}
\put(10,23){\line(1,0){50}}
\put(7,0){\makebox(0,0)[cc]{$A_1$}}

\put(63,0){\makebox(0,0)[cc]{$A_2$}}

\put(63,46){\makebox(0,0)[cc]{$A_3$}}

\put(7,46){\makebox(0,0)[cc]{$A_4$}}

\put(37,26){\makebox(0,0)[cc]{$A$}}

\put(50,33){\makebox(0,0)[cc]{$\Pi_1$}}

\put(50,13){\makebox(0,0)[cc]{$\Pi_4$}}

\put(25,13){\makebox(0,0)[cc]{$\Pi_3$}}

\put(25,33){\makebox(0,0)[cc]{$\Pi_2$}}

\end{picture}
\caption{Division of~$\Pi$ into four rectangles}
\label{rectangle}
\end{figure}

And, finally, our third result refines Proposition~\ref{p6.2} in a particular case. 

\begin{lemma} \label{l5.4}
Let $Y$ be an arbitrary metric space, let $\KK\subset H$ be a compact subset, let $(y,\omega)\mapsto V^y(\omega)$ be a random finite set, and let
$$
\aA_\omega^y=\bigcup_{v\in V^y(\omega)}(v+\KK).
$$
Then there is a random finite set $(\delta,y,\omega)\mapsto U_{\delta,y}(\omega)$ with the underlying space $(0,1]\times Y\times H$ such that~\eqref{6.8} holds, and 
\begin{align} 
\ln\bigl(\# U_{\delta,y}(\omega)\bigr)&\le \ln(\# V^{y}(\omega)\bigr)+\HH_{\delta/2}(\KK,H),
\label{5.21}\\
d^s\bigl(U_{\delta,y_1}(\omega_1),U_{\delta,y_2}(\omega_2)\bigr)
&\le d^s\bigl(V^{y_1}(\omega_1),V^{y_2}(\omega_2)\bigr).
\label{5.22}
\end{align}
\end{lemma}

\begin{proof}
Applying Lemma~\ref{l6.1} to the random compact set $\delta\mapsto \delta\KK$ with the underlying space~$(0,1]$, we construct a random finite set $\delta\mapsto U_\delta$ such that 
$$
d^s(\delta\KK,U_\delta)\le\delta^2,\quad
\ln(\#U_\delta)\le \HH_{\delta^2/2}(\delta\KK,H)=\HH_{\delta/2}(\KK,H). 
$$
It is straightforward to see that the random set 
$$
U_{\delta,y}(\omega)=\delta^{-1}U_\delta+V^y(\omega)
=\{\delta^{-1}u+v:u\in U_\delta, v\in V^y(\omega)\}
$$
possesses all required properties. 
\end{proof}

As is clear from the proof, if~$V^y(\omega)$ does not depend on~$\omega$ for some $y=y_0$, then the random set~$U_{\delta,y_0}(\omega)$ constructed in Lemma~\ref{l5.4} is also independent of~$\omega$. 

\begin{proof}[Proof of Lemma~\ref{l5.3}]
Given a point~$A\in\Pi$, we divide the rectangle~$\Pi$ into four smaller rectangles~$\Pi_i$ (see Figure~\ref{rectangle}). 
It is easy to prove that the functions 
\begin{equation} \label{5.20}
\theta_i(A)=\frac{\area(\Pi_i)}{\area(\Pi)}, \quad 1\le i\le 4,
\end{equation}
possess the required properties. 
\end{proof}

\subsection{Image of random compact sets}
\label{s5.02}
\begin{proposition} \label{p5.6}
Let $X$ and~$Y$ be Polish spaces, let $(\Omega,\FF)$ be a measurable space, let $\{\KK_\omega,\omega\in\Omega\}$ be a random compact set in~$X$, and let $\psi_\omega:X\to Y$ be a family of continuous mappings such that, for any $u\in X$, the mapping $\omega\mapsto \psi_\omega(u)$ is measurable from~$\Omega$ to~$Y$. Then $\{\psi_\omega(\KK_\omega), \omega\in\Omega\}$ is a random compact set in~$Y$. 
\end{proposition}

\begin{proof}
Let us set fix $u\in Y$ and define a function $F_u:\Omega\to \R$ by 
$$
F_u(\omega)=d_Y\bigl(u,\psi_\omega(\KK_\omega)\bigr)
=\inf_{v\in \KK_\omega}d_Y(u,\psi_\omega(v)). 
$$
We need to prove that this function is measurable. Let~$\{u_k\}\subset X$ be a dense sequence. Repeating the argument used in the proof of Lemma~\ref{l6.1}, we can construct a sequence of random finite sets~$\KK_\omega^n\subset\{u_k\}$ such that
\begin{gather}
d_X^s(\KK_\omega,\KK_\omega^n)\le\frac1n, \label{7.31}\\
\KK_\omega^n=\{u_k,I_k^n(\omega)=1\}, \label{7.32}
\end{gather}
where $I_k^n:\Omega\to\{0,1\}$, $k,n\ge1$, are measurable functions. It follows from~\eqref{7.31} that 
$$
F_u(\omega)=\liminf_{n\to\infty}F_u^n(\omega), \quad
F_u^n(\omega)=d_Y\bigl(u,\psi_\omega(\KK_\omega^n)\bigr). 
$$
Thus, it suffices to establish the measurability of~$F_u^n$. To this end, note that, in view of~\eqref{7.32}, we have
$$
F_u^n(\omega)=\inf_{v\in\KK_\omega^n}d_Y(u,\psi_\omega(v))
=\inf_{k\ge1}\frac{d_Y(u,\psi_\omega(u_k))}{I_k^n(\omega)}. 
$$
This relation readily implies the required property. 
\end{proof}

\subsection{Kolmogorov--\v Centsov theorem}
\label{s6.2}
The Kolmogorov--\v Centsov theorem provides a sufficient condition for H\"older-continuity of trajectories of a random process. We shall need the following qualitative version of that result. 

\begin{theorem} \label{t5.6}
Let $X$ be a Banach space and let $\{\xi_t, 0\le t\le T\}$ be an $X$-valued random process with almost surely continuous trajectories that is defined on a probability space $(\Omega,\FF,\IP)$ and satisfies the inequality
\begin{equation} \label{7.50}
\E\,\|\xi_t-\xi_s\|_X^{2p}\le C_p|t-s|^p\quad\mbox{for any $t,s\in[0,T]$, $p\ge1$},
\end{equation}
where $C_p>0$ is a constant not depending on~$t$ and~$s$. Then for any $\gamma\in(0,1/2)$ there is a constant $K_\gamma>0$ and an almost surely positive random variable $t_\gamma$ such that
\begin{align}
\|\xi_t(\omega)-\xi_s(\omega)\|_X&\le K_\gamma|t-s|^\gamma
\quad\mbox{for $|t-s|\le t_\gamma(\omega)$}, \label{7.51}\\
\E\,t_\gamma^{-q}&<\infty\quad\mbox{for any $q\ge1$}.\label{7.52} 
\end{align}
\end{theorem}

\begin{proof}[Sketch of the proof]
We repeat the argument used in Section~2.2.B of~\cite{KS1991}. Without loss of generality, we can assume that $T=1$. Let us fix any $\gamma\in(0,1/2)$ and introduce the events
$$
\Omega_n^{(k)}=\bigl\{\omega\in\Omega: 
\|\xi_{k/2^n}(\omega)-\xi_{(k-1)/2^n}(\omega)\|_X
\ge 2^{-\gamma n}\bigr\}, 
\quad \Omega_n=\bigcup_{k=1}^{2^n}\Omega_n^{(k)},
$$
where $n\ge1$ and $1\le k\le 2^n$. It follows from~\eqref{7.50} and the Chebyshev inequality that 
$$
\IP\bigl(\Omega_n^{(k)}\bigr)\le C_p2^{-np(1-2\gamma)}. 
$$
Summing up over $k=1,\dots,2^n$, we derive
$$
\IP(\Omega_n)\le C_p2^{-n\alpha_p}, \quad  \alpha_p=-1+p(1-2\gamma).
$$ 
Choosing $p\ge1$ so large that $\alpha_p>0$ and applying the Borel--Cantelli lemma, we construct an almost surely finite random integer $n_0\ge1$ such that $\omega\notin\Omega_n$ for $n\ge n_0(\omega)$ and $\omega\in\Omega_{n_0-1}$ if $n_0(\omega)\ge2$. In particular, we have
\begin{equation} \label{7.53}
 \|\xi_{k/2^n}(\omega)-\xi_{(k-1)/2^n}(\omega)\|_X
 \ge 2^{-\gamma n}
 \quad\mbox{for $n\ne n_0(\omega)$, $k=1,\dots,2^n$}. 
\end{equation}
As is shown in the proof of Theorem~2.8 of~\cite[Chapter~2]{KS1991}, inequality~\eqref{7.53} implies~\eqref{7.51} with $K_\gamma=2/(1-2^{-\gamma})$ and $t_0=2^{-n_0}$. Thus, the theorem will be proved if we show that $\E\,2^{qn_0}<\infty$ for any $q\ge1$. 

To this end, note that $\{n_0=m\}\subset\Omega_{m-1}$ for any $m\ge2$. It follows that
$$
\E\,2^{qn_0}\le 2^q+\sum_{m=2}^\infty 2^{qm}\IP(\Omega_{m-1})
\le 2^q+C_p\sum_{m=2}^\infty 2^{qm-\alpha_p(m-1)}. 
$$
Choosing $p\ge1$ so large that $\alpha_p>q$, we see that the series on the right-hand side of the above inequality converges. 
\end{proof}

Note that one can rewrite~\eqref{7.51} and~\eqref{7.52} in the form
$$
\|\xi_t(\omega)-\xi_s(\omega)\|_X\le C_\gamma(\omega)\,|t-s|^\gamma
\quad t,s\in[0,T],
$$
where $C_\gamma$ is a random variable with finite moments. 
We now apply the above result to establish a time-regularity property for the process~$U^\omega$ defined in the beginning of Section~\ref{s5.2}.

\begin{proposition} \label{c5.7}
For any $\gamma\in(0,1/2)$ and any $T>0$ there is a random variable $C_{\gamma,T}>0$ all of whose moments are finite such that
\begin{equation} \label{6.30}
\|U(t)-U(s)\|_2\le C_{\gamma,T}\,|t-s|^\gamma
\quad \mbox{for $t,s\in[-T,T]$}. 
\end{equation}
\end{proposition}

\begin{proof}
In view of the remark following the proof of Theorem~\ref{t5.6}, it suffices to check that~$U$ satisfies inequality~\eqref{7.50} with $[0,T]$ replaced by~$[-T,T]$ and $X=H^1$. Since~$U$ is stationary, we can assume that $s=0$. Equation~\eqref{4.7} implies that
$$
U(t)-U(0)=\int_0^ta\Delta U(r)\,dr+\zeta(t), 
$$
whence, applying the H\"older inequality, it follows that
$$
\|U(t)-U(0)\|_2^{2p}\le 2^{2p-1}
\biggl\{|t|^p\biggl(\int_0^t\|a\Delta u\|_2^{2}dr\biggr)^p+\|\zeta\|_2^{2p}\Biggr\}. 
$$
Using~\eqref{4.9}, we see that the mean value of first term on the right-hand side can be estimated by $C|t|^{p}$. Thus, the required inequality will be established if we show that
\begin{equation} \label{6.31}
\E\,\|\zeta\|_2^{2p}\le C_p|t|^p. 
\end{equation}
To this end, we note that $\|\zeta\|_2^2=\sum_jc_j^2\beta_j^2(t)$, where $c_j=b_j\lambda_j$. 
The monotone convergence theorem and the Burkholder inequality (see Theorem~2.10 in~\cite{HH1980}) imply that 
\begin{align*}
\E\,\|\zeta\|_2^{2p}&=\lim_{n\to\infty}\E\biggl(\sum_{j=1}^n c_j^2\beta_j^2(t)\biggr)^p
\le C_1\lim_{n\to\infty}\E\,\biggl|\sum_{j=1}^n c_j\beta_j(t)\biggr|^{2p}\\
&= C_2(p)\lim_{n\to\infty}\Bigl(|t|\sum_{j=1}^nc_j^2\Bigr)^p
\le C_3(p)\,|t|^p,
\end{align*}
where we used the fact that $\sum_j c_j\beta_j(t)$ is a zero-mean Gaussian random variable with variance $t\sum_jc_j^2$. This proves~\eqref{6.31} and completes the proof of the proposition. 
\end{proof}

\addcontentsline{toc}{section}{Bibliography}
%\bibliography{references}

\begin{thebibliography}{FGMZ04}

\bibitem[BV92]{BV1992}
A.~V. Babin and M.~I. Vishik, \emph{Attractors of {E}volution {E}quations},
  North-Holland Publishing, Amsterdam, 1992.

\bibitem[CDF97]{CDF-97}
H.~Crauel, A.~Debussche, and F.~Flandoli, \emph{Random attractors}, J. Dynam.
  Differential Equations \textbf{9} (1997), no.~2, 307--341.

\bibitem[CF94]{CF-94}
H.~Crauel and F.~Flandoli, \emph{Attractors for random dynamical systems},
  Probab. Theory Related Fields \textbf{100} (1994), no.~3, 365--393.

\bibitem[CF98]{CF-1998}
\bysame, \emph{Additive noise destroys a pitchfork bifurcation}, J. Dynam.
  Differential Equations \textbf{10} (1998), 259--274.

\bibitem[CV02]{CV2002}
V.~V. Chepyzhov and M.~I. Vishik, \emph{{A}ttractors for {E}quations of
  {M}athematical {P}hysics}, AMS Coll. Publ., vol.~49, AMS, Providence, 2002.

\bibitem[CVZ12]{CVZ-2012}
V.~Chepyzhov, M.~Vishik, and S.~Zelik, \emph{Regular attractors and their
  non-autonomous perturbations}, Mat. Sb. (N.S.) (2012).

\bibitem[DZ92]{DZ1992}
G.~{Da Prato} and J.~Zabczyk, \emph{Stochastic {E}quations in {I}nfinite
  {D}imensions}, Cambridge University Press, Cambridge, 1992.

\bibitem[EFNT94]{EFNT1994}
A.~Eden, C.~Foias, B.~Nicolaenko, and R.~Temam, \emph{Exponential attractors
  for dissipative evolution equations}, RAM: Research in Applied Mathematics,
  vol.~37, Masson, Paris, 1994.

\bibitem[EMZ05]{EMZ-2005}
M.~Efendiev, A.~Miranville, and S.~Zelik, \emph{Exponential attractors and
  finite-dimensional reduction for non-autonomous dynamical systems}, Proc.
  Roy. Soc. Edinburgh Sect. A \textbf{135} (2005), 703--730.

\bibitem[FGMZ04]{FGMZ-2004}
P.~Fabrie, C.~Galusinski, A.~Miranville, and S.~Zelik, \emph{Uniform
  exponential attractors for a singularly perturbed damped wave equation},
  Discrete Contin. Dyn. Syst. \textbf{10} (2004), no.~1-2, 211--238.

\bibitem[Fla94]{flandoli-1994}
F.~Flandoli, \emph{Dissipativity and invariant measures for stochastic
  {N}avier--{S}tokes equations}, NoDEA Nonlinear Differential Equations Appl.
  \textbf{1} (1994), no.~4, 403--423.

\bibitem[HH80]{HH1980}
P.~Hall and C.~C. Heyde, \emph{Martingale {L}imit {T}heory and {I}ts
  {A}pplication}, Academic Press, New York, 1980.

\bibitem[KS91]{KS1991}
I.~Karatzas and S.~E. Shreve, \emph{Brownian {M}otion and {S}tochastic
  {C}alculus}, Springer-Verlag, New York, 1991.

\bibitem[KS12]{KS2012}
S.~Kuksin and A.~Shirikyan, \emph{Mathematics of {T}wo-{D}imensional
  {T}urbulence}, Cambridge University Press, Cambridge, 2012.

\bibitem[Lor86]{lorentz1986}
G.~G. Lorentz, \emph{Approximation of {F}unctions}, Chelsea Publishing Co., New
  York, 1986.

\bibitem[MZ08]{MZ-2008}
A.~Miranville and S.~Zelik, \emph{Attractors for dissipative partial
  differential equations in bounded and unbounded domains}, Handbook of
  differential equations: evolutionary equations. {V}ol. {IV}, North-Holland,
  Amsterdam, 2008, pp.~103--200.

\bibitem[Str93]{stroock1993}
D.~Stroock, \emph{Probability. {A}n {A}nalytic {V}iewpoint}, Cambridge
  University Press, Cambridge, 1993.

\bibitem[SY02]{SY2002}
G.~R. Sell and Y.~You, \emph{Dynamics of {E}volutionary {E}quations}, Springer
  Verlag, Berlin, 2002.

\bibitem[Wal82]{walters1982}
P.~Walters, \emph{An {I}ntroduction to {E}rgodic {T}heory}, Graduate Texts in
  Mathematics, vol.~79, Springer-Verlag, New York, 1982.
\end{thebibliography}
%\bibliographystyle{amsalpha}

\def\cprime{$'$} \def\polhk#1{\setbox0=\hbox{#1}{\ooalign{\hidewidth
  \lower1.5ex\hbox{`}\hidewidth\crcr\unhbox0}}} \def\cprime{$'$}
  \def\polhk#1{\setbox0=\hbox{#1}{\ooalign{\hidewidth
  \lower1.5ex\hbox{`}\hidewidth\crcr\unhbox0}}} \def\cprime{$'$}
  \def\cprime{$'$} \def\cprime{$'$} \def\cprime{$'$}
\providecommand{\bysame}{\leavevmode\hbox to3em{\hrulefill}\thinspace}
\providecommand{\MR}{\relax\ifhmode\unskip\space\fi MR }
\providecommand{\MRhref}[2]{%
  \href{http://www.ams.org/mathscinet-getitem?mr=#1}{#2}
}
\providecommand{\href}[2]{#2}

\end{document}